\documentclass{theoretics}

\ThCSauthor[1]{Ferenc Bencs}{ferenc.bencs@gmail.com}[0000-0002-2554-5838]
\ThCSauthor[2]{Khallil Berrekkal}{k.berrekkal@uva.nl}[0009-0005-8532-1413]
\ThCSauthor[2]{Guus Regts}{guusregts@gmail.com}[0000-0002-7813-4952]

\ThCSaffil[1]{Centrum Wiskunde \& Informatica, Netherlands}
\ThCSaffil[2]{University of Amsterdam, Netherlands}

\ThCSthanks{Ferenc Bencs is supported by the Netherlands Organisation of Scientific Research (NWO): VI.Veni.222.303. \\ Khallil Berrekkal and Guus Regts are supported by the Netherlands Organisation of Scientific Research (NWO): VI.Vidi.193.068 }
\ThCSshortnames{F. Bencs, K. Berrekkal, G. Regts}
\ThCSshorttitle{Deterministic Approximate Counting of Colorings}
\ThCSyear{2026}
\ThCSarticlenum{1}
\ThCSreceived{Dec 10, 2024}
\ThCSrevised{Aug 25, 2025}
\ThCSaccepted{Sep 29, 2025}
\ThCSpublished{Jan 13, 2026}
\ThCSdoicreatedtrue
\ThCSkeywords{Proper colorings, anti-ferromagnetic Potts model, partition function, approximate counting, zero-freeness}

\title{Deterministic Approximate Counting of Colorings with fewer than \texorpdfstring{$2\Delta$}{2 Delta} Colors via Absence of Zeros}

\usepackage[T1]{fontenc}

\makeatother

\usepackage{mathtools}
\usepackage{bm}
\usepackage{eurosym}
\usepackage{colortbl}
\usepackage{thm-restate}
\usepackage{tikz-cd}
\usepackage{float}
\usetikzlibrary{decorations.pathreplacing}
\usetikzlibrary{shapes.geometric}
\usetikzlibrary{patterns}
\usepackage{standalone}
\usepackage{comment}
\usepackage{epstopdf}

\newcommand*{\R}{\mathbb{R}}

\newcommand*{\C}{\mathbb{C}}

\newcommand*{\prob}{\mathbb{P}}
\newcommand*{\pr}{\mathbb{P}}

\begin{document}
\maketitle

\begin{abstract}
Let $\Delta,q\geq 3$ be integers.
We prove that there exists $\eta\geq 0.002$ such that if $q\geq (2-\eta)\Delta$, then there exists an open set $\mathcal{U}\subset \mathbb{C}$ that contains the interval $[0,1]$ such that for each $w\in \mathcal{U}$ and any graph $G=(V,E)$ of maximum degree at most $\Delta$, the partition function of the anti-ferromagnetic $q$-state Potts model evaluated at $w$ does not vanish. 
This provides a (modest) improvement on a result of Liu, Sinclair, and Srivastava, and breaks the $q=2\Delta$-barrier for this problem.

As a direct consequence we obtain via Barvinok's interpolation method a \emph{deterministic} polynomial time algorithm to approximate the number of proper $q$-colorings of graphs of maximum degree at most $\Delta$, provided $q\geq (2-\eta)\Delta$.
\end{abstract}

\section{Introduction}
The algorithmic problem of designing an algorithm to (approximately) compute the number of $q$-colorings of a graph $G$ has received a lot of interest in the past thirty years.
The main challenge is to design for each pair of positive integer $\Delta,q$ such that $q\geq \Delta+1$ an algorithm that on input of an $n$-vertex graph $G$ of maximum degree at most $\Delta$ and $\varepsilon>0$ outputs the number of $q$-colorings of $G$ within an $\exp(\varepsilon)$ relative error in time \textsc{poly($n/\varepsilon$)}. 
We note that if $q< \Delta$, approximating the number of proper $q$-colorings of a graph of maximum degree $\Delta$ is \textsc{NP-hard}~\cite{GSVhardnesscoloring} (provided $q$ is even), even for triangle-free graphs.
Thus far the only nontrivial cases for which such an algorithm is known to exist correspond to $\Delta=3$ and $q\geq4$~\cite{Luetalfptasforcubic}.
Below we will say more about the status of this problem in general.

Over the past thirty years several algorithmic approaches have been developed to approximately count the number of proper colorings (as well as for several other counting problems) including variations of the celebrated MCMC method~\cite{chen2019improved,Jerrumsimple}, the correlation decay approach~\cite{BanGar,GKfptas,Weitz} and the interpolation method~\cite{Barbook,BDPR21,LSS2Delta,PatReg17}. 
The latter two approaches yield \emph{deterministic} algorithms as opposed to the MCMC-based method.

The interpolation method is based on the existence of a zero-free region for an associated family of polynomials whose evaluations count the number of proper colorings. 
In the present paper this role is taken by the \emph{partition function of the Potts model}, which we will introduce shortly. 
The main focus of the present paper is to provide an improved zero-free region for this polynomial, which, as a direct corollary, yields efficient deterministic approximation algorithms for counting proper colorings. 

\paragraph{The partition function of the Potts model.}
For a graph $G=(V,E)$, a positive integer $q$ and $w\in \mathbb{C}$, the \emph{partition function of the Potts model} is defined as 
\begin{equation}\label{eq:def pf}
Z_G(q;w):=\sum_{\phi:V\to [q]} w^{m(\phi)},  
\end{equation}
where $[q]:=\{1,\ldots,q\}$, and $m(\phi)$ denotes the number of \emph{monochromatic} edges, i.e., the number of edges $\{u,v\}$ such that $\phi(u)=\phi(v)$.
We note that $Z_G(q;0)$ equals the number of proper $q$-colorings of $G$.
In statistical physics, one usually takes $w>0$ parameterized as $w=e^{\beta J}$, where $J$ denotes the coupling constant, and $\beta$ the inverse temperature. Here $J<0$ corresponds to the anti-ferromagnetic case, while $J>0$ corresponds to the ferromagnetic case.

Partly motivated by the classical Lee-Yang~\cite{lee1952statistical} and Fisher~\cite{fisher1965nature} approach to phase transitions, there is an interest in the location of the complex zeros of $Z_G(q;w)$, both in terms of the variable $w$ \cite{BDPR21,feldmann1998study,LSS2Delta,LSSfisher,patel2023near} and in the variable $q$~\cite{BiggsShrock,Borgszeros,Roederqplane,ChangShrock,Roederchromatic,FerProc,JS09,jenssen2023improved,shrock2000exact,Sokalzeros} (for the latter, one has to extend the partition function of the Potts model to the partition function of the random cluster model, where $q$ can also take non-integer values).

\paragraph{Absence of zeros in computer science.}
More recently, there has been an increasing interest in understanding the location of these complex zeros from the perspective of computer science and probability theory. 
This interest comes from the fact that zero-free regions for the partition functions of models, such as the Potts model, yield efficient deterministic approximation algorithms~\cite{Barbook,PatReg17}, rapid mixing of the associated Glauber dynamics~\cite{anarisector,vigodastability}, (local) central limit theorems for associated random variables~\cite{localcentrallimit,LPRScentrallimit,Michelencentrallimit}, and decay of correlations~\cite{gamarnikabsence,regts2021absence}.
In particular, open sets $\mathcal{U}$ containing the interval $[0,1]$ such that $Z_G(q;w)\neq 0$ for all graphs of maximum degree at most $\Delta$ and $w\in \mathcal{U}$ are of interest, since via Barvinok's interpolation method~\cite{Barbook,PatReg17} they yield efficient algorithms for approximately computing the number of proper $q$-colorings of these graphs, a notorious problem in computer science~\cite{BanGar,BDPR21,carlson2025flip,chen2019improved,Friezevigodasurvey,GKfptas,Hayes2003,Jerrumsimple,LSS2Delta,Luetalfptasforcubic,lu2013improved,Vigoda}.
It is a folklore conjecture that such algorithms exist provided $q\geq \Delta +2$~\cite{Friezevigodasurvey}.
So far, this has only been proved for $q\geq \tfrac{11}{6}\Delta$ by Vigoda~\cite{Vigoda} in case one allows the algorithm to use randomness. This bound on $q$ stood for nearly $20$ years until Chen et al.~\cite{chen2019improved} improved this to $q\geq (\tfrac{11}{6}-\varepsilon)\Delta$ with $\varepsilon\approx 10^{-5}$. Very recently, a more substantial improvement due to Carlson and Vigoda~\cite{carlson2025flip} appeared, which states that one can take $\varepsilon\geq 0.024$.
For \emph{deterministic} algorithms, the existence of such an algorithm is only known\footnote{After the first posting of the present paper to the arXiv, it was shown in~\cite{chen2024deterministic} that the randomized algorithm of Carlson and Vigoda can in fact be derandomized.} when $q\geq 2\Delta$ by a result of Liu, Sinclair, and Srivastava~\cite{LSS2Delta}. 

\paragraph{The \texorpdfstring{$2\Delta$}{2Delta} bound of  Liu, Sinclair, and Srivastava~\cite{LSS2Delta}.}
The deterministic algorithm of Liu, Sinclair, and Srivastava~\cite{LSS2Delta} is based on Barvinok's interpolation method. Their main contribution lies in proving a zero-free region for the partition function of the Potts model. 
They essentially prove the following more general statement, allowing them to deduce zero-free regions from probabilistic statements.

\begin{itemize}
\item[$(\star)$]\emph{Let $\mathcal{G}$ be a class of graphs of maximum degree at most $\Delta$. Suppose that $q$ is such that for all $w\in [0,1]$ and for any rooted graph $(G,v)$ ($G\in \mathcal{G}$) in which some of its vertices are precolored, when drawing a random coloring from the Potts model with parameter $w$ the marginal probability that $v$ gets color $i$ (assuming that $i$ is not used on the neighbors of $v$) is bounded by $\tfrac{1}{d+2}$, where $d$ denotes the number of neighbors of $v$ that are not precolored.
Then there exists an open set $\mathcal{U}$ containing $[0,1]$ such that $Z_G(q;w)\neq 0$ for all $w\in \mathcal{U}$ and $G\in \mathcal{G}$.}
\end{itemize}
It is not difficult to see that for $q\geq 2\Delta$ this condition is satisfied for the class of all graphs of maximum degree at most $\Delta$, and hence this immediately gives the desired zero-freeness and approximation algorithm via the interpolation method.
For triangle free graphs, the condition is satisfied provided $q\geq 1.7633\Delta +\beta$ where $\beta$ is an absolute constant, see~\cite{LSS2Delta} for the precise statement.
It is easy to see that for $q<2\Delta$ there are examples of graphs where the condition in $(\star)$ is not satisfied. 
Unfortunately, the proof of $(\star)$ as given in~\cite{LSS2Delta} is somewhat technical, making it difficult to see how to push the bounds on $q$ below $2\Delta$.

\paragraph{Our contributions.}
This brings us to the contributions of the present paper.
One of our contributions is that we give a new proof of the existence of the zero-free region for $q\geq 2\Delta$, which is shorter and arguably more transparent and less technical.
Secondly, we are able to take advantage of the local structure around the root vertex for graphs where the condition in $(\star)$ is not met, and thereby provide a modest improvement on the result of Liu, Sinclair, and Srivastava~\cite{LSS2Delta}. 

\begin{theorem}\label{thm:main}
There exists a constant $\eta\geq 0.002$ such that for all integers $\Delta\geq 3$ and $q\geq (2-\eta)\Delta$ there exists an open set $\mathcal{U}\subset \mathbb{C}$ containing the interval $[0,1]$ such that for each $w\in \mathcal{U}$ and graph $G$ of maximum degree at most $\Delta$, $Z_G(q,w)\neq 0$.
\end{theorem}

As a direct corollary, we have the following result breaking the $q=2\Delta$ barrier for designing efficient \emph{deterministic} approximation algorithms for counting proper colorings based on absence of zeros. 

\begin{corollary}\label{cor:main alg}
There exists a constant $\eta\geq 0.002$ such that for all integers $\Delta\geq 3$ and $q\geq (2-\eta)\Delta$ and $w\in [0,1]$ there exists a deterministic algorithm which given an $n$-vertex graph of maximum degree at most $\Delta$ and $\varepsilon>0$ computes a number $\xi$  satisfying
    \[
        e^{-\varepsilon}\le \frac{Z_G(q,w)}{\xi}\le  e^{\varepsilon}    
    \]
    in time polynomial in $n/\varepsilon$.
\end{corollary}

We note that this is indeed a direct corollary of Theorem~\ref{thm:main} using Barvinok's interpolation method~\cite{Barbook} in combination with the improvement due to Patel and the last author of the present paper~\cite{PatReg17}.
An explanation of how this fits the framework of~\cite{PatReg17} can be found in the proof of~\cite[Corollary 1]{BDPR21} and, therefore, we omit a proof here.

Another consequence of Theorem~\ref{thm:main} is that for $q\geq (2-\eta)\Delta$ and $w\in (0,1)$ the random variable defined as the number of monochromatic edges in a random sample from the Potts model (with parameters $q$ and $w$) on a graph of maximum degree $\Delta$ satisfies a (local) central limit theorem. 
This follows almost directly from the results in~\cite{Michelencentrallimit,localcentrallimit,kawasakidynamics} and we refer the interested reader to these papers for the relevant details.

\paragraph{Organization and conventions.}
The remainder of the paper is dedicated to proving Theorem~\ref{thm:main}, first for $\eta=0$ and later for $\eta>0$. 
In the next section we give a detailed technical outline of our approach, at the end of which the reader may find an overview of the remainder of the paper.

While our arguments can for example be extended to list-coloring, we opt for focusing on just colorings so as to limit any technical overhead and hopefully making the proof more transparent. We comment on the extension to list-colorings and other possible extensions in Section~\ref{sec:conclusion}.

Although we could slightly improve our lower bound on $\eta$, we decided to stick to a bound with only three decimal places. 
Our current established bounds do not seem to allow us to replace the lower bound of $0.002$ by $0.003$. 
We comment on possible approaches for improvement in Section~\ref{sec:conclusion}.

\section{Outline of approach and setup}\label{sec:setup}
In this section, we give a detailed outline of our approach, which is inspired by~\cite{Barbook,BDPR21,LSS2Delta} and uses several concepts developed in these papers.
The main idea is to use induction to prove a result about partition functions of graphs with certain pre-colored vertices.
To carry out the induction, we will need good control over how changing the color of a vertex affects the partition function.
We next introduce some definitions that will be used throughout the paper and that will facilitate the discussion of the proof outline.

\subsection{Definitions and expanded theorem statement}
Let $q>0$ be an integer, let $G=(V,E)$ be a graph, let $S\subset V$ and let $\phi:S\to [q]:=\{1,\ldots,q\}$. 
We call the triple $(G,S,\phi)$ a \emph{partially $q$-colored graph}. 
Often we will just say that $G$ is a partially $q$-colored graph, omitting the reference to $S$ and $\phi$. 
A vertex $v\in S$ will be called \emph{pinned}; any vertex not contained in $S$ will be called a \emph{free} vertex. 
We say that a color $j\in [q]$ is \emph{blocked at $v$} for a vertex $v\in V$ if either $v$ is pinned, or if $v$ has a neighbor in $S$ which is assigned the color $j$ by $\phi$, otherwise color $j$ is called \emph{free at $v$}.

The partition function of the Potts model of a partially $q$-colored graph $G=(G,S,\phi)$ is defined as
\begin{align}\label{eq:def pf partial}
    Z_G(w):=\sum_{\substack{\psi:V(G)\to [q]\\ \psi_{|S}=\phi}} w^{m(\psi)}.
\end{align}
Note that we remove the $q$ from the argument of $Z_G$ compared to~\eqref{eq:def pf} since $q$ is already implicit in $G$.

Given a partially $q$-colored graph $(G,S,\phi)$, 
we can always assume that each pinned vertex $v\in S$ is a leaf of $G$ (i.e. has degree $1$) without changing the partition function, by iteratively replacing each pinned vertex $v$ by $d=\deg(v)$ copies of it, $v_1,\ldots,v_d$, connecting each of them to a unique neighbor of $v$. 
We denote by $\mathcal{G}_{\Delta,q}^\bullet$ the set of pairs $(G,v)$, where $G$ is a connected partially $q$-colored graph of maximum degree $\Delta$ and $v$ is a free vertex of $G$ and where the pinned vertices of $G$ are all leaves and form an independent set.
For such a pair let us define the vector ${\bf c}_{G,v}\in\mathbb{N}_{\geq 0}^q$, where the $i$th coordinate, $c_{G,v;i}$, denotes the number of pinned neighbors of $v$ that are colored with color $i$. We refer to ${\bf c}_{G,v}$ as the vector of blocked colors at $v$.
As a convention, we will write vectors in boldface, while entries of vectors are denoted in plain typeface.

If $w\geq 0$ and if $Z_G(w)\neq 0$ there is an associated probability measure, $\pr_{G,w}$, on the collection of all colorings $\psi:V\to [q]$ that coincide with $\phi$ on $S$, whose probability mass function is defined by
\begin{align*}
    \mu_{G,w}(\psi):=\frac{w^{m(\psi)}}{Z_G(w)}.
\end{align*}
We will use capital letters to denote random variables. In particular, we denote the probability that vertex $v$ is assigned color $j$ when sampling a coloring from this distribution by $\pr_{G,w}[\Phi(v)=j]$. 
When $w$ is clear from the context, we often simply write $\pr_{G}$ instead of $\pr_{G,w}$.

For $w\in \mathbb{C}$ and a free vertex $v$ of a partially $q$-colored graph $G$ we consider the ratio
\[
\tilde{R}_{G,v;i,j}(w):=\frac{Z^i_{G,v}(w)}{Z^j_{G,v}(w)}, 
\]
as a rational function in $w$. Here $Z^j_{G,v}(w)$ denotes the sum~\eqref{eq:def pf partial} restricted to those $\psi$ that assign color $j$ to the vertex $v$. Note that $\tilde{R}_{G,v;i,j}(w)$ only depends on the connected component of $G$ that contains $v$.

To prove that $Z_G(w)\neq 0$ for some $w\in \mathbb{C}$ and a partially $q$-colored graph, it suffices to inductively show that $Z^j_{G,v}(w)\neq 0$ for some color $j\in [q]$ and that the ratios $\tilde{R}_{G,v;i,j}(w)$ ($i\in [q]$) pairwise make a small angle.
In~\cite{BDPR21} this is done via a direct recursive approach, by showing that these ratios are trapped in a certain set in the complex plane. 
In~\cite{LSS2Delta} this is done via a clever indirect approach by showing that for fixed $w\in [0,1]$ and $\tilde w$ close enough to $w$, we have that the perturbed ratios $\tilde{R}_{G,v;i,j}(\tilde w)$ are close to the original ratios $\tilde{R}_{G,v;i,j}(w)$ and in particular lie close to the real axis and hence make a pairwise small angle.
This is also the approach we follow in the present paper.
To do so, we will change coordinates and work with \emph{log-ratios}. Define
\begin{align*}
R_{G,v;i,j}(\tilde w):=\log(\tilde{R}_{G,v;i,j}(\tilde w)),
\end{align*}
under the implicit assumption that $Z_{G,v}^i(\tilde w)$ and $Z_{G,v}^j(\tilde w)$ are both not equal to $0$ and where we fix the branch of the logarithm that is real valued on the positive real line.

For $(G,v)\in \mathcal{G}^{\bullet}_{\Delta,q}$ we denote by $(\overline{G},v)\in \mathcal{G}^{\bullet}_{\Delta,q}$ the rooted partially $q$-colored graph obtained from $(G,v)$ by removing all pinned neighbors of $v$ from $G$. 
We often just write $\overline{G}$ in case $v$ is clear from the context.
The \emph{free degree} of a vertex is the number of free neighbors of that vertex.
We next state an expanded version of our main theorem.

\begin{theorem}\label{thm:main expanded}
There exists a constant $\eta\geq 0.002$ such that for all integers $\Delta\geq 3$, $q\geq (2-\eta)\Delta$ there exist $\varepsilon_1>0$ and $\varepsilon_2>0$, such that if $(G,v)\in \mathcal{G}^{\bullet}_{\Delta,q}$ where $v$ has free degree at most $\Delta-1$, then for any colors $i,j\in [q]$, any $w\in [0,1]$ and any $\tilde w\in B(w,\varepsilon_1)$,
\begin{align}\label{eq:bounded log ratios}
Z_{G}(\tilde w)\neq 0\quad\text{and}\quad  |R_{\overline{G},v;i,j}(\tilde w)-R_{\overline{G},v;i,j}(w)|\leq \varepsilon_2.
\end{align}
\end{theorem}
\begin{remark}\label{rem:dependence}
It follows from our proof that we can take $\varepsilon_1\geq C \Delta^{-4}$ for some constant $C>0$.
For $\eta=0$ this improves on the size of the zero-free region given by Liu, Sinclair and Srivastava~\cite{LSS2Delta} who proved a zero-free region around $[0,1]$ of width $C'\Delta^{-16}$ for some constant $C'>0$.

The consequences of this improvement for the running time for the algorithm in Corollary~\ref{cor:main alg} are limited though. 
The running time can be seen to be bounded by $(n/\varepsilon)^{O(\log(\Delta q)\exp(O(\Delta^4 )))}$  by combining~\cite[Lemma 2.2.3]{Barbook} and~\cite{PatReg17}.
\end{remark}
Below, we give an outline of our proof of this result; the actual proof can be found in Section~\ref{sec:proof eta>0}. 
First, we use it to deduce Theorem~\ref{thm:main}.

\begin{proof}[Proof of Theorem~\ref{thm:main}]
Let $\varepsilon_1$ be as in the statement of Theorem~\ref{thm:main expanded}, fix $w\in [0,1]$ and let $\tilde{w}\in B(w,\varepsilon_1)$.
By Theorem~\ref{thm:main expanded}, it suffices to prove that $Z_G(q;\tilde w)\neq 0$ if $G$ is $\Delta$-regular.
To cover the case where all the vertices of $G$ have degree exactly $\Delta$, and hence are not pinned, we will use the symmetry of the model between the colors.
We may further assume that $G$ is connected since the partition function factors over connected components.

First we claim that the partition function $Z_{G,v}^1(q;\tilde{w})$ is non-zero. 
Indeed, $Z_{G,v}^1(q;\tilde{w})$ is equal to the partition function of the partially $q$-colored graph $H$ obtained from $G$ by replacing $v$ with vertices $v_1,\dots,v_{\Delta}$ each of them colored with color $1$ and where each $v_i$ is connected to a unique neighbor of $v$ in $G$. 
We next claim that $Z_H(\tilde w)\neq 0$. Since each component of $H$ has a vertex of free degree $\Delta-1$, namely a neighbor of some $v_i$, and since by construction we have that the pinned vertices of $H$ form an independent set and are all leaves, by Theorem~\ref{thm:main expanded} we indeed have $Z_H(\tilde w)\neq 0$ (because the partition function is multiplicative over the components of $H$).

Now, since each vertex in $G$ is free, it follows by symmetry that
\[
Z_{G,v}^1(q;\tilde{w}) = Z_{G,v}^2(q;\tilde{w}) = \ldots = Z_{G,v}^q(q;\tilde{w}).
\]
Therefore,
\[
Z_{G}(q;\tilde{w}) = \sum_{i \in [q]} Z_{G,v}^i(q;\tilde{w}) = q \cdot Z_{G,v}^q(q;\tilde{w})=qZ_H(\tilde w) \neq 0.
\]
\end{proof}

\subsection{Outline of proof and more definitions}\label{subsec:outline}
To prove Theorem~\ref{thm:main expanded}, we need to show that the difference between $R_{G,v;\ell_1,\ell_2}(w)$ and $R_{G,v;\ell_1,\ell_2}(\tilde w)$ is smaller than $\varepsilon_2$ for each $(G,v)\in \mathcal{G}^\bullet_{\Delta,q}$ and any pair of colors $\ell_1,\ell_2$.
We do this by induction on the number of free vertices, by expanding the log-ratios of $(G,v)$ as a function applied to log-ratios of partially $q$-colored graphs obtained from $G$ with fewer free vertices. 
We will next describe some of the technical ingredients that were also used in some form in~\cite{BDPR21,LSS2Delta} and the main new ideas of our proof, after which we will give an overview of the remainder of the paper.

To be able to control the difference between the log-ratios inductively, a certain \emph{telescoping} procedure is crucial for us.
Fix two distinct colors $\ell_1,\ell_2\in [q]$.
Choose an ordering of the neighborhood $N(v)$ of $v$, $v_1,\ldots,v_{\deg(v)}$.
Let $\hat{G}_i$ be the partially $q$-colored graph obtained from $G-v$ by adding a leaf to each vertex $v_j$ with $j\neq i$ such that for $j<i$ that leaf is colored with color $\ell_2$ and for $j>i$ it is colored with color $\ell_1$; the leaf connected to $v_i$ is free and is denoted by $\hat{v}_i$.
Assume that $Z_{\hat{G}_i,\hat{v}_i}^\ell(\tilde w)\neq 0$ for all $\ell\in [q]$ and $i=1,\ldots, \deg(v)$.
Then by standard properties of the logarithm,
\begin{equation}\label{eq:telescoping before}
R_{G,v;\ell_1,\ell_2}=\sum_{i=1}^{\deg(v)}R_{\hat{G}_i,\hat{v}_i;\ell_1,\ell_2}.
\end{equation}
See Figure~\ref{fig:telescoping} for a proof by pictures of this identity.

\begin{figure}[t]
  \centering
  \includegraphics[width=0.9\textwidth]{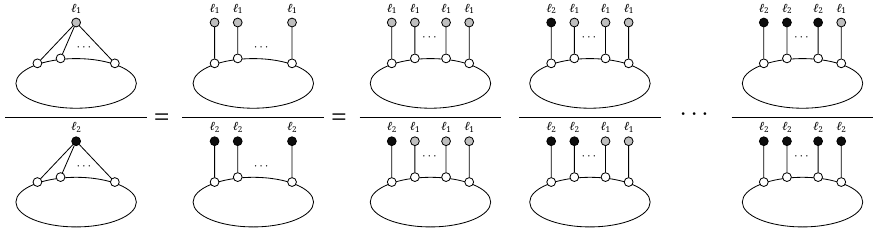}
 
\caption{Pictorial depiction of how the ratio $Z^{\ell_1}_{G,v}(w) \big/Z^{\ell_2}_{G,v}(w)$ is expressed as a telescoping product of the ratios $Z^{\ell_1}_{\hat{G}_i,\hat{v}_i}(w) \big/ Z^{\ell_2}_{\hat{G}_i,\hat{v}_i}(w)$.}
\label{fig:telescoping}
\end{figure}
\begin{cfigure}[t]
    \centering
\includegraphics[width=0.7\textwidth]{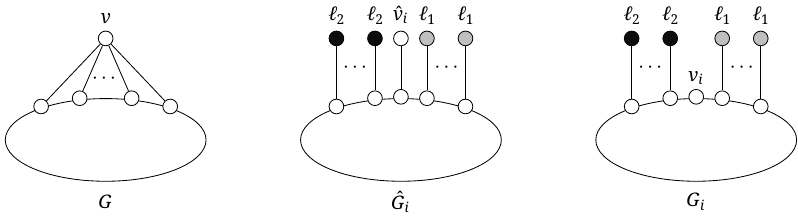}
    \caption
    {An illustration of a graph $(G,v)$ and the graphs $(\hat G_i,\hat v_i)$ and $(G_i,v_i)$ appearing in the telescoping procedure.}
    \label{fig:modified_hat_G_i_graph}
\end{cfigure}

Let us denote by $G_i$ the partially $q$-colored graph obtained from $\hat{G}_i$ by removing the vertex $\hat{v}_i$.
We will say that the graphs $G_i$ are obtained from $G$ via the \emph{telescoping procedure} with respect to the colors $\ell_1$ and $\ell_2$ and the ordering of $N(v)$. See Figure~\ref{fig:modified_hat_G_i_graph} for an illustration of this.

For concreteness, we will continue the discussion for the case $\ell_1=1$ and $\ell_2=q$. 
By symmetry, we may of course always relabel the colors so that this is without loss of generality.

Define for ${\bf c}\in \mathbb{N}_{\geq 0}^q$, ${\bf x}\in \mathbb{C}^{q-1}$ and $\tilde w\in \mathbb{C}$,
\begin{align}
P_{\bf c}({\tilde w},{\bf x})={\tilde w}^{c_1+1}e^{x_1}+\sum_{j=2}^{q-1} {\tilde w}^{c_j}e^{x_j}+{\tilde w}^{c_q}, \label{eq:def P}
\\
Q_{\bf c}({\tilde w},{\bf x})={\tilde w}^{c_1}e^{x_1}+\sum_{j=2}^{q-1} {\tilde w}^{c_j}e^{x_j}+{\tilde w}^{c_q+1}.\label{eq:def Q}
\end{align}
We tend to omit the subscript $\bm{c}$ if it is clear from the context. 
We now define the vectors ${\bf R}^i(\tilde w)\in \mathbb{C}^{q-1}$ by $R^i_{j}(\tilde w)=R_{\overline{G}_i,v_i;j,q}({\tilde w})$ for $j=1,\ldots,q-1$ (implicitly assuming that these log-ratios are well-defined), and ${\bf c}^i={\bf c}_{G,v_i}$. 
Then we observe that 
\begin{align*}
P_{{\bf c}^i}({\tilde w},{\bf R}^i(\tilde w))= \frac{Z_{{\hat G}_i,\hat{v}_i}^1(\tilde w)}{Z_{\overline{G}_i,v_i}^q(\tilde w)}  \text{ and }
Q_{{\bf c}^i}({\tilde w},{\bf R}^i(\tilde w))= \frac{Z_{{\hat G}_i,\hat{v}_i}^q(\tilde w)}{Z_{\overline{G}_i,v_i}^q(\tilde w)},  
\end{align*}
and therefore by equation~\eqref{eq:telescoping before} we have 
\begin{equation}\label{eq:telescoping after}
R_{G,v;1,q}=\sum_{i=1}^{\deg(v)} \log\left(\frac{P_{\bf c}({\tilde w},{\bf R}^i)}{Q_{\bf c}({\tilde w},{\bf R}^i)}\right).
\end{equation}

Next define the function $F_{{\tilde w},{\bf c}}:\mathbb{C}^{q-1}\setminus \{{\bf x}\mid P_{\bf c}({\tilde w,\bf x})=0 \text{ or } Q_{\bf c}(\tilde w,{\bf x})= 0\} \to \mathbb{C}$ by
\begin{equation}\label{eq:def Fwd}
F_{{\tilde w},\bf c}({\bf x})=\log\left(\frac{P_{\bf c}({\tilde w},{\bf x})}{Q_{\bf c}({\tilde w},{\bf x})}\right).
\end{equation}
Using this function we can succinctly express~\eqref{eq:telescoping after} as 
\begin{align}
R_{G,v;1,q}(\tilde w)=\sum_{i=1}^{\deg(v)}F_{\tilde w,{\bf c}^i}({\bf R}^i(\tilde w)).\label{eq:function log ratio}
\end{align}

We  now use~\eqref{eq:function log ratio} to express the difference between $R_{G,v;1,q}(w)$ for $w\in [0,1]$ and its perturbation $R_{G,v;1,q}(\tilde w)$ (for $\tilde w$ near $w$) as follows
\begin{align}
&R_{G,v;1,q}(w)-R_{G,v;1,q}(\tilde w)= \sum_{i=1}^{\deg(v)} \left (F_{w,{\bf c}^i}({\bf R}^i(w))-F_{\tilde w,{\bf c}^i}({\bf R}^i(\tilde w)\right)=\nonumber
\\
& \sum_{i=1}^{\deg(v)} \left (F_{w,{\bf c}^i}({\bf R}^i(w))-F_{w,{\bf c}^i}({\bf R}^i(\tilde w))\right)+\left(F_{w,{\bf c}^i}({\bf R}^i(\tilde w))-F_{\tilde w,{\bf c}^i}({\bf R}^i(\tilde w))\right ).\label{eq:difference}
\end{align}
It is not hard to see that we can make $|F_{w,{\bf c}^i}({\bf R}^i(\tilde w))-F_{\tilde w,{\bf c}^i}({\bf R}^i(\tilde w))|$ arbitrarily small by choosing $\tilde w$ close enough to $w$ by continuity of $F_{w,{\bf c}^i}$ as a function of $w$, provided the vectors ${\bf R}^i(\tilde w)$ lie in a bounded set. We will formally verify this in Section~\ref{sec:analytic bounds}.

Bounding $F_{w,{\bf c}^i}({\bf R}^i(w))-F_{w,{\bf c}^i}({\bf R}^i(\tilde w))$ is much more work and to do so we will use the gradient of $F_{w,{\bf c}^i}$. 
We have
\begin{align}
  F_{w,{\bf c}^i}&({\bf R}^i(w))-F_{w,{\bf c}^i}({\bf R}^i(\tilde w))=\label{eq:difference in terms of gradient}
  \\
  & \int_{0}^{1}\langle \nabla F_{w,{\bf c}^i}\left(t{\bf R}^i(w) +(1-t){\bf R}^i(\tilde w)\right),{\bf R}^i(w)-{\bf R}^i(\tilde w)\rangle  dt,\nonumber
\end{align}
where $\langle \cdot,\cdot\rangle$ denotes the standard inner product on $\mathbb{C}^{q-1}.$

This motivates us to investigate the gradient of $F_{w,{\bf c}}$. 
For this purpose let us define for a partially $q$-colored graph $H$ with a free vertex $v$ and a color $\ell\in[q]$ the partially $q$-colored graph $H^{+\ell}$ by attaching a pinned leaf of color $\ell$ to $v$. 
For $w\in [0,1]$ the vector $\textbf{P}_{H,v}(w)$ in $\mathbb{R}^{q-1}$ is defined by 
\begin{equation}\label{eq:define P}
{P}_{H,v;j}(w)=\pr_{H^{+1},w}[\Phi(v)=j]-\pr_{H^{+q},w}[\Phi(v)=j].
\end{equation}

The next lemma tells us that the gradient of $F_{w,{\bf c}_{H,v}}$ evaluated at a real log-ratio vector is exactly the vector ${\bf P}_{H,v}(w)$.

\begin{lemma}\label{lem:gradient for real w}
Let $w\in[0,1]$.
Let $(H,v)\in\mathcal{G}_{\Delta,q}^\bullet$ and let ${\bf c}={\bf c}_{H,v}$ be the vector of blocked colors at $v$.  
Let ${\bf R}\in \mathbb{R}^{q-1}$ be the vector defined by $R_{j}=R_{\overline{H},v;j,q}(w)$.
Then 
\[
\nabla F_{w,{\bf c}}({\bf R})={\bf P}_{ H,v}(w).
\]
\end{lemma}
\begin{proof}
Writing $P=P_{\bf c}(w,{\bf x})$ and $Q=Q_{\bf c}(w,{\bf x})$, we have $F_{w,\bf c}({\bf x})=\log\left(\frac{P}{Q}\right)$ and therefore by standard rules of partial derivatives,
\begin{align*}
    \frac{\partial F_{w,{\bf c}}}{\partial x_j}&= \frac{Q}{P} \cdot\frac{Q\tfrac{\partial P}{x_j}-P\tfrac{\partial Q}{\partial x_j}}{Q^2}\\
    &= \frac{Qw^{c_j}e^{x_j} + \delta_{1,j}Qw^{c_j}e^{x_j}(w-1) - Pw^{c_j}e^{x_j}  }{PQ}\\
   & =\frac{(w-1)w^{c_1}e^{x_1}}{P}\delta_{1,j}+\frac{w^{c_j}e^{x_j}(Q-P)}{PQ}  \\
   & =\frac{(w-1)w^{c_1}e^{x_1}}{P}\delta_{1,j}+\frac{w^{c_j}e^{x_j}}{P}-\frac{w^{c_j}e^{x_j}}{Q}.
\end{align*}
Noting that $w^{c_j}e^{R_j}=\tfrac{Z_{H,v}^j(w)}{Z_{H,v}^q(w)}$ and hence 
\[
P=w\frac{Z_{H,v}^1(w)}{Z_{H,v}^q(w)}+\sum_{j=2}^q\frac{Z_{H,v}^j(w)}{Z_{H,v}^q(w)}=\frac{Z_{H^{+1},v}(w)}{Z_{H,v}^q(w)},
\]
and similarly,
\[
Q=w\frac{Z_{H,v}^q(w)}{Z_{H,v}^q(w)}+\sum_{j=1}^{q-1}\frac{Z_{H,v}^j(w)}{Z_{H,v}^q(w)}=\frac{Z_{H^{+q},v}(w)}{Z_{H,v}^q(w)},
\]
we see that
\[
  \frac{\partial F_{w,{\bf c}}}{\partial x_j}({\bf R})=\pr_{H^{+1},w}[\Phi(v)=1]-\pr_{H^{+q},w}[\Phi(v)=1]
\]
i.e. $\nabla F_{w,{\bf c}}({\bf R})={\bf P}_{H,v}(w)$, as desired.
\end{proof}

To use~\eqref{eq:difference in terms of gradient}, we actually need to understand the gradient of $F_{w,{\bf c}}$ evaluated at the vector $ tR_{\overline{G}_i,v_i;j,q}(\tilde{w}) \allowbreak + (1-t)R_{\overline{G}_i,v_i;j,q}(\tilde{w})$ for some small perturbation $\tilde{w}$ of $w$. 
However, it is not difficult to see that for $\tilde{w}$ small enough and $q>\Delta+1$, this gradient can be made arbitrarily close to ${\bf P}_{{G}_i,v_i}(w)$ independent of the graph.
We verify this formally in Section~\ref{sec:analytic bounds}.

So to bound \eqref{eq:difference in terms of gradient}, the essential ingredient is to bound the absolute value of the inner product of the vector ${\bf P}_{{G}_i,v_i}(w)$ with the vector ${\bf R}_{\overline{G}_i,v_i;j,q}(w)-{\bf R}_{\overline{G}_i,v_i;j,q}(\tilde w)$ for each $i=1,\ldots,d$.
This motivates us to develop bounds on the marginal probability of the root vertex, which we will do in Section~\ref{sec:prob bounds 1} and Section~\ref{sec:more marginal bounds}.

With the bounds from Section~\ref{sec:prob bounds 1} and Section~\ref{sec:analytic bounds} we give a proof of Theorem~\ref{thm:main} for $\eta=0$ in Section~\ref{sec:proof eta=0}. 
We do this for two reasons. 
First of all, it makes it easier to see the structure of the technically more involved proof for the proof when $\eta>0$, and secondly, it gives an alternative and arguably more transparent proof of the result of Liu, Sinclair, and Srivastava~\cite{LSS2Delta}. 
The ingredients that we use are pretty much the same as in~\cite{LSS2Delta}; the main difference is that we do not bound the real part and imaginary part of the log-ratio difference separately as is done in~\cite{LSS2Delta}, but our novel contribution is to make use of the symmetry between the colors to bound the absolute value of the inner product of the vector ${\bf P}_{ G,v}(w)$ with the vector ${\bf R}^i(w)-{\bf R}^i(\tilde w)$ in terms of a marginal probability of the root vertex times the maximum of 
\begin{equation}\label{eq:ratio difference for induction}
|{R}_{\overline{G}_i,v_i,j,q}(w)-{R}_{\overline{G},v;j,q}(\tilde w)| \textrm{ and } |{R}_{\overline{G}_i,v_i;j,1}(w)-{R}_{\overline{G},v;j,1}(\tilde w)|,
\end{equation}
over all $j=1,\ldots,q$ (see Lemma~\ref{lem:replace by prob}), which in our proof we show is bounded by $\varepsilon_2$, by induction.
Since for $q\geq 2\Delta$ this marginal probability is easily seen to be bounded by $1/\Delta$ (see Lemma~\ref{lem:prob basic} below), this allows to conclude that $|R_{G,v;1,q}(w)-R_{G,v;1,q}(\tilde w)|$ is again bounded by $\varepsilon_2$. See Section~\ref{sec:proof eta=0} for the details.

In Section~\ref{sec:proof eta>0} we finally give a proof of Theorem~\ref{thm:main expanded} for $\eta>0$.
Here we build on the approach for $\eta=0$ and carefully make use of the structure of the local neighborhood of the vertex $v$ to show that either both terms in~\eqref{eq:ratio difference for induction} are actually smaller than $\varepsilon_2$ or that we can use sharper bounds on the marginal probability obtained in Section~\ref{sec:more marginal bounds} that are valid in a more restricted setting.

In an appendix, we collect an alternative proof of a proposition found in Section~\ref{sec:analytic bounds} that is more hands on, but has the advantage that it gives concrete dependencies on how small $\varepsilon_1$ should be in terms of $\Delta$.
We conclude with some questions and remarks in Section~\ref{sec:conclusion}.

\section{Basic bounds on marginal probabilities of the root vertex}\label{sec:prob bounds 1}
Let $\Delta$ and $q>\Delta+1$ be positive integers and let $w\in [0,1]$.
Let $G$ be a partially $q$-colored graph of maximum degree at most $\Delta$ and let $v$ be a vertex of $G$ of degree $d$. 
To prove our main result we will need bounds on the marginal probability of the root vertex $\pr_{G,w}[\Phi(v)=j]$. 
In this section we collect upper and lower bounds on this quantity that are well known in the literature, but we provide proofs for the sake of completeness and because we will build on these proofs later on.
We start with an upper bound.

\begin{lemma}\label{lem:prob basic}
Let $\Delta$ and $q>\Delta+1$ be positive integers.
Assume that $G$ is a partially $q$-colored graph of maximum degree at most $\Delta$ and let $v$ be a vertex of $G$ of degree $d$  and free degree $f$. 
Denote by $b$ the number of blocked colors at $v$.
Then for any $w\in [0,1]$ any color $j$,
\begin{align*}
    \pr_{G,w}[\Phi(v)=j]\leq \frac{w^{c_j}}{q-(f+b)+(f+b)w},
\end{align*}
where $c_j={c}_{G,v;j}$. 
In particular,
\begin{align*}
    \pr_{G,w}[\Phi(v)=j]\leq \frac{w^{c_j}}{q-d+dw}.
\end{align*}
\end{lemma}
\begin{proof}
We can expand $\pr_{G,w}[\Phi(v)=1]$ over the colorings of the neighbors of $v$.
This yields
\begin{align}\label{eq:expand prob}
    \pr_{G,w}[\Phi(v)=j]=\sum_{\kappa:N(v)\to [q]}\pr_{G,w}[\Phi(v)=j\mid E_\kappa]\pr_{G,w}[E_\kappa],
\end{align}
where $E_\kappa$ denotes the event that the random coloring $\Phi$ agrees with $\kappa$ on the neighbors $N(v)$ of $v$. 
Note that when $w>0$, we always have that $\pr_{G,w}[E_\kappa]>0$, while if $w=0$ we only sum over those $\kappa$ for which $\pr_{G,w}[E_\kappa]>0$.

Now to bound  $\pr_{G,w}[\Phi(v)=j]$, we simply need a bound on $\pr_{G,w}[\Phi(v)=j\mid E_\kappa]$ for any $\kappa$ for which $\pr_{G,w}[E_\kappa]>0$.
Let us fix such a $\kappa$ and denote by $d_i$ for $i\in [q]$ the number of neighbors of $v$ colored with color $i$.
Then by the Markov property and using the convention that $w^0=1$, 
\begin{align}
\pr_{G,w}[\Phi(v)=j\mid E_\kappa]&=\frac{w^{d_j}}{\sum_{i=1}^{q} w^{d_i}}=\frac{w^{d_j}}{q-\sum_{i=1}^q(1-w^{d_i})}\label{eq:prob E_kappa equal}
\\
&\leq \frac{w^{c_j}}{q-(1-w)(b+f)},   \label{eq:prob E_kappa}
\end{align}
where $c_j=c_{G,v;j}$ denotes the number of pinned neighbors with color $j$ in $G$, $b$ denotes the number of blocked colors at $v$, and $f$ the free degree of $v$.
\end{proof}

\begin{remark}
Note that the lemma above is essentially tight for $w=0$ and $j$ such that $c_j=0$. Indeed consider a vertex whose neighborhood is a clique of size $\Delta-1$ such that for each neighbor color $j$ is blocked.  
\end{remark}

Next we provide a lower bound on the marginal probability.
\begin{lemma}\label{lem:prob basic lower}
Let $\Delta$ and $q$ be positive integers, and let $q\ge(1+ \alpha)\Delta+1$ for some $\alpha>0$.
Assume that $G$ is a partially $q$-colored graph of maximum degree at most $\Delta$ and let $v$ be a vertex of $G$. 
Then, for any $w\in [0,1]$ and any color $j\in [q]$ not appearing on the neighbors of $v$, we have
\begin{align*}
    \pr_{G,w}[\Phi(v)=j] \geq \frac{1}{e^{1/\alpha} q}.
\end{align*}
\end{lemma}

\begin{proof}
To prove a lower bound on the marginal probability of the root vertex getting a free color $j$, we again look at~\eqref{eq:expand prob} and note that we can lower bound this by
\begin{align*}
    \pr_{G,w}[\Phi(v)=j] \geq \sum_{\substack{\kappa:N(v)\to [q] \\ j \notin \kappa(N(v))}}\pr_{G,w}[\Phi(v)=j\mid E_\kappa]\pr_{G,w}[E_\kappa].  
\end{align*}
Given $\kappa:N(v)\to [q]$ such that $\pr_{G,w}[E_\kappa]\neq 0$ and $\kappa$ does not use color $j$ we can lower bound $\pr_{G,w}[\Phi(v)=j\mid E_\kappa]$ by $1/q$ by~\eqref{eq:prob E_kappa equal}.
This implies
\begin{align*}
     \pr_{G,w}[\Phi(v)=j] \geq 
     \frac{1}{q}  \pr_{G,w}[j \notin \Phi(N(v))].
\end{align*}
Let $u_1, \ldots, u_{d}$ be the free neighbors of $v$ in $G$. 
We can then express
\begin{align}
    \pr_{G,w}[j \notin \Phi(N(v))] &=    \pr_{G,w}[1 \notin \Phi( \{u_1, \ldots u_{d-1} \}) \mid \Phi(u_d)\neq j] \pr_{G,w}[\Phi(u_d)\neq j] \nonumber
    \\
    &= \prod_{i=1}^d  \pr_{G,w}[ \Phi(u_i) \neq j \mid E_i] \nonumber
    \\
    &=  \prod_{i=1}^d \left( 1- \pr_{G,w}[ \Phi(u_i) = j \mid E_i] \right), \label{eq:product expand}
\end{align}
where $E_i$ denotes the event that the random coloring $\Phi$ does not assign color $j$ to the vertices $u_{i+1}, \ldots, u_d$, which has a nonzero probability since $q>\Delta+1$.
Arguing as above (regardless of the event $E_i$, we can expand the marginal probability as a sum over $\kappa$ as in~\eqref{eq:expand prob}, but now we only sum over those $\kappa$ that do not assign color $j$ to the vertices $u_{i+1},\ldots,u_d$) we can bound 
\begin{align*}
 \pr_{G,w}[ \Phi(u_i) = j \mid E_i] \leq \frac{1}{q-\Delta+\Delta w}.
 \end{align*}
Therefore, assuming $q-\Delta\geq \alpha \Delta+1$ for some $\alpha>0$, we obtain the lower bound 
\begin{align}
\pr_{G}[j \notin \Phi(N(v))] \ge \left( 1 - \frac{1}{q-\Delta+\Delta w} \right)^d \geq \left(\left(1- \frac{1}{\alpha \Delta+1}  \right)^{\alpha \Delta}\right)^{1/\alpha} \ge e^{-1/\alpha}.   \label{eq:lower bound on not using 1}
\end{align}
We conclude that $\pr_{G,w}[\Phi(v)=1]$ is lower bounded by $\tfrac{1}{e^{1/\alpha} q}$.
\end{proof}

\section{The behavior of \texorpdfstring{$F_{w,{\bf c}}$}{F} and \texorpdfstring{$\nabla F_{w,{\bf c}}$}{nabla F}}\label{sec:analytic bounds}

As indicated in Section~\ref{sec:setup}, we need to show that behavior of $F_{w,{\bf c}}$ (as defined in~\eqref{eq:def Fwd}) and $\nabla F_{w,\bf {c}}$ as a function of $w$ is not too wild.

Let $\Delta,q$ be positive integers such that $q\geq (1+\alpha)\Delta+1$ for some $\alpha>0$.
We call a vector ${\bf c}\in \mathbb{N}_{\geq 0}^q$ a \emph{valid color vector} if there exists a rooted graph $(G,v)\in \mathcal{G}^\bullet_{\Delta,q}$ such that ${\bf c}={\bf c}_{G,v}$.
Now fix a valid color vector $\bf c$ and let us define the set 
\begin{equation}
    \mathcal{R}_{\bf c}= \left\{ \left(w,\left(R_{\overline{G},v;j,q}(w)\right)_{j=1,\dots,q-1} \right)~|~(G,v)\in \mathcal{G}^\bullet_{\Delta,q}, {\bf c}={\bf c}_{G,v}, w\in [0,1]\right\}.
\end{equation}
Define $F:\mathcal{R}_{\bf c}\to \mathbb{R}\subset\mathbb{C}$ by $F(w,{\bf R})=F_{w,\bf c}({\bf R})$ for $(w,{\bf R)}\in \mathcal{R}_{\bf c}$.
\begin{lemma}\label{lem: ratio bounds}
Let $\Delta,q$ be positive integers such that $q\geq (1+\alpha)\Delta+1$ for some $\alpha>0$, and let ${\bf c}$ be a valid color vector. 
Then the set $F(\mathcal{R}_{\bf c})$ is bounded.
\end{lemma}
\begin{proof}
For any $(w,{\bf R})\in \mathcal{R}_{\bf c}$, $j\in [q-1]$ we have
\begin{equation*}
\exp(R_j)=\frac{Z^j_{\overline{G},v}(w)}{Z^q_{\overline{G},v}(w)}=\frac{\pr_{\overline{G},w}[\Phi(v)=j]}{\pr_{\overline{G},w}[\Phi(v)=q]},
\end{equation*}
and therefore by Lemma~\ref{lem:prob basic} and Lemma~\ref{lem:prob basic lower},

\begin{equation*}
    \frac{q-\Delta}{qe^{1/\alpha}}\le \exp(R_j)\le \frac{qe^{1/\alpha}}{q-\Delta}.
\end{equation*}

Then, since at least $q-\Delta$ entries of ${\bf c}$ are equal to $0$ and $w\in [0,1]$,
\begin{align*}
\frac{(q-\Delta)^2}{qe^{1/\alpha}}\leq & P_{\bf c}(w,{\bf R})\le \frac{q^2 e^{1/\alpha}}{q-\Delta},
\\
\frac{(q-\Delta)^2}{qe^{1/\alpha}}\leq  & Q_{\bf c}(w,{\bf R})\le  \frac{q^2 e^{1/\alpha}}{q-\Delta},
\end{align*}
where we recall that $ P_{\bf c}$ and  $Q_{\bf c}$ are defined in~\eqref{eq:def P}  and~\eqref{eq:def Q}, respectively.
It thus follows that 
\[
\frac{(q-\Delta)^3}{q^3e^{2/\alpha}}\le \left|\frac{P_{\bf c}(w,{\bf R})}{Q_{\bf c}(w,{\bf R})}\right|\leq \frac{q^3e^{2/\alpha}}{(q-\Delta)^3}.
\]
Since $F_c=\log(P_{\bf c}/Q_{\bf c})$, this implies that $F(\mathcal{R}_{\bf c})$ is a bounded set, as desired.
\end{proof}

\begin{proposition}
\label{prop:gradient and w}
Let $\alpha>0$ and let $q,\Delta$ be positive integers such that $q\geq (1+\alpha) \Delta+1$.
Then for any $\varepsilon\in (0,1)$ there exists a $\delta >0$ such that the following holds.
Let $(G,v)\in\mathcal{G}_{\Delta,q}^\bullet$, and let $w\in[0,1]$.
Let ${\bf R}\in \mathbb{R}^{q-1}$ be the vector defined by $R_{j}=R_{\overline{G},v;j,q}(w)$.
Then
\begin{itemize}
    \item[(i)]
    if ${\bf x} \in \mathbb{C}^{q-1}$ and $\|{\bf R}-{\bf x}\|_\infty\le \delta$, then
\[
\|{\bf P}_{G,v}(w)-\nabla F_{w,{\bf c}}({\bf x})\|_1\le \varepsilon;
\]
\item[(ii)] if ${\bf x}\in\mathbb{C}^{q-1}$ and $\|{\bf R}-{\bf x}\|_\infty\le \delta$ and $|\tilde w-w|\leq \delta$, then 
\[
    |F_{w,{\bf c}_{G,v}}({\bf x})-F_{\tilde w,{\bf c}_{G,v}}({\bf x})|\le \varepsilon.
\]
\end{itemize}
\end{proposition}

\begin{remark}\label{rem: refer to appendix}
Below we give a concise proof using a compactness argument, but which does not display how $\delta$ depends on $\varepsilon$.
We refer the reader to Appendix ~\ref{appendix:additional computations} for an explicit proof of this proposition, where we also show $\delta$ can be taken of the form $\min\left\{ C(\alpha)\varepsilon, C(\alpha) /\Delta \right\}$ for some constant $C(\alpha)$ that only depends on $\alpha$.
\end{remark}

\begin{proof}
Let us fix a valid color vector ${\bf c}$.
Since by the previous lemma we know that $F$ is bounded on $\mathcal{R}_c$, we know that there exists 
an $r>0$ such that $F$ is well-defined and bounded on the compact set
\[ 
    \mathcal{D}_{\bm{c}}:=\overline{B(\mathcal{R}_{\bm{c}},r)}\subseteq \mathbb{C}\times \mathbb{C}^{q-1},
\]
where we view $\mathbb{C}\times \mathbb{C}^{q-1}$ as the direct sum of the normed space $\mathbb{C}$ with the absolute value and $\mathbb{C}^{q-1}$ with the $\infty$-norm, and where $\overline{B(\mathcal{R}_{\bm{c}},r)}$ denotes the closure of the set of points that have distance at most $r$ to $\mathcal{R}_{\bm{c}}$.

Now $\nabla F_{w,\bm{c}}:\mathcal{D}_{\bm{c}}\to \mathbb{C}^{q-1}$ (defined by $(w,{\bf x})\mapsto \nabla F_{w,\bf c}({\bf x})$) is a continuous function on a compact set. 
Therefore, for any $\varepsilon>0$ there exists an $r\ge \delta_{1,\bf c}>0$, such that 
if $(w,{\bf x}),(w,{\bf y})\in\mathcal{D}_{\bf c}$ and $\|{\bf x}-{\bf y}\|_1\le \delta_{1,\bf c}$, then
\[
\|\nabla F_{w,\bf c}({\bf x})-\nabla F_{w,\bf c}({\bf y})\|\le \varepsilon.
\]
Similarly, there exists $\delta_{2,{\bf c}}$ such that for any $(w,{\bf x}), (w',\bf y)\in \mathcal{D}_{\bf c}$ we have
\[
|F_{w,{\bf c}}({\bf x})-F_{w',{\bf c}}({\bf y})|\le \varepsilon.
\]
There are finitely many choices of $\bf c$ for a given $q$ and $\Delta$, therefore we can choose $\delta':=\min\{\delta_{1,\bf c},\delta_{2,{\bf c}}\mid {\bf c}  \text{ valid color vector}\}>0$  and set $\delta=\delta'/2$.

For this choice of $\delta$ we have for the given $(w,{\bf R})\in \mathcal{R}_{{\bf c}_{G,v}}$ and any ${\bf x}$ such that $\|{\bf x}-{\bf R}\|_\infty\leq\delta$,
\[
 \|{\bf P}_{G,v}-\nabla F_{w,{\bf c}_{G,v}}({\bf x})\|_1=\|\nabla F_{w,{\bf c}_{G,v}}({\bf R})-\nabla F_{w,{\bf c}_{G,v}}({\bf x})\|_1\leq \varepsilon,
\]
and similarly, for any $(\tilde w,{\bf x})$ such that $|\tilde w-w|\leq \delta$ and $\|{\bf R}-{\bf x}\|_\infty\leq \delta$ (so that $\|(w,{\bf R})-(\tilde w,{\bf x})\|\leq \delta'$) we have
\[
|F_{w,{\bf c}_{G,v}}({\bf R})-F_{\tilde w,{\bf c}_{G,v}}({\bf x})|\le \varepsilon.
\]
This finishes the proof.
\end{proof}

The next lemma allows us in combination with Proposition~\ref{prop:gradient and w} to bound~\eqref{eq:difference in terms of gradient}.
\begin{lemma}\label{lem:replace by prob}
Let $G$ be a partially $q$-colored graph and let $v$ be a free vertex of $G$.
Let $w\in [0,1]$ and let ${\bf P}_{G,v}$ be defined as in~\eqref{eq:define P}. 
Let ${\bf x}\in \mathbb{C}^{q-1}$ and let $\widehat{{\bf x}}$ be defined by $\widehat{x}_1=-x_1$ and by $\widehat{x}_j=x_j-x_1$ for $j=2,\dots,q-1$.
Then
\[
|\langle{\bf P}_{G,v},{\bf x}\rangle|\le \left\{\begin{array}{ll}
(1-w)\pr_{G^{+1},w}[\Phi(v)=q] \cdot \|{\bf x }\|_{\infty} &\textrm{if $\pr_{G,w}[\Phi(v)=1]\le \pr_{G,w}[\Phi(v)=q]$}\\
(1-w)\pr_{G^{+q},w}[\Phi(v)=1] \cdot \|\widehat{{\bf x}}\|_{\infty} &\textrm{if $\pr_{G,w}[\Phi(v)=1]>\pr_{G,w}[\Phi(v)=q]$}
\end{array}\right..
\]
\end{lemma}
\begin{proof}
First of all, note that possibly after multiplication of the coordinates of ${\bf x}$ by $e^{i\vartheta}$ for some $\vartheta\in \mathbb{R}$, we may assume that
\[
|\langle{\bf P}_{G,v},{\bf x}\rangle|=\langle{\bf P}_{G,v},{\bf x}\rangle=\langle {\bf P}_{G,v},\Re({\bf x})\rangle,
\]
since ${\bf P}_{G,v}$ is real valued.
We may therefore restrict to real vectors ${\bf x}$.

Next observe that the first coordinate of ${\bf P}_{G,v}$ is non-positive, since
\begin{align*}
    P_{G,v;1}(w)&=\frac{wZ_{G,v}^{1}(w)}{Z_G(w)-(1-w)Z_{G,v}^1(w)}-\frac{Z_{G,v}^{1}(w)}{Z_G(w)-(1-w)Z_{G,v}^q(w)}\\
    &=\frac{(1-w)Z_{G,v}^1(w)(-wZ_{G,v}^q(w)-Z_G(w)+Z_{G,v}^1(w))}{(Z_G(w)-(1-w)Z_{G,v}^1(w))(Z_G(w)-(1-w)Z_{G,v}^q(w))}\\
    &\le \frac{-Z_{G,v}^1(w)(1-w)(w+1)Z_{G,v}^q(w)}{(Z_G(w)-(1-w)Z_{G,v}^1(w))(Z_G(w)-(1-w)Z_{G,v}^q(w))}\\
    &\le 0.
\end{align*}
Note that with similar reasoning it follows that
\begin{equation}\label{eq:positive}
\pr_{G^{+1},w}[\Phi(v)=q]-\pr_{G^{+q},w}[\Phi(v)=q]\geq 0.
\end{equation}
Next we distinguish the two cases. 

If $\pr_{G,w}[\Phi(v)=1]\le \pr_{G,w}[\Phi(v)=q]$, then
\[
    Z_{G,v}^1(w)\le Z_{G,v}^q(w),
\]
and for $j>1$ we have
\begin{align*}
    P_{G,v;j}&=\frac{Z_{G,v}^{j}(w)}{Z_G(w)-(1-w)Z_{G,v}^1(w)}-\frac{Z_{G,v}^{j}(w)}{Z_G(w)-(1-w)Z_{G,v}^q(w)}\le  0.
\end{align*}
So in this case each entry of ${\bf P}_{G,v}$ is non-positive and hence 
\[
|\langle{\bf P}_{G,w},{\bf x}\rangle|\leq \|{\bf x}\|_\infty\langle{\bf P}_{G,w},{-\bf 1}\rangle=\|{\bf x}\|_\infty(\pr_{G^{+1}}[v=q]-\pr_{G^{+q}}[v=q]).
\]
We next note that since $Z^1_{G,v}(w)\le Z^q_{G,v}(w)$ we have

\[
\sum_{j=2}^{q-1}Z^j_{G,v}(w)+wZ^q_{G,v}(w)+Z^1_{G,v}(w)\leq w Z^1_{G,v}(w)+\sum_{j=2}^{q-1}Z^j_{G,v}(w)+Z^q_{G,v}(w),
\]
and therefore, 
\begin{equation}\label{eq:(1-w) mult bound}
\pr_{G^{+1},w}[\Phi(v)=q]-\pr_{G^{+q},w}[\Phi(v)=q]\le (1-w)\pr_{G^{+1},w}[\Phi(v)=q],   
\end{equation}
proving the first case.

If $\pr_{G,w}[\Phi(v)=1]>\pr_{G,w}[\Phi(v)=q]$, the first entry of ${\bf P}_{G,v}$ is non-positive while all other entries are non-negative. 
Hence we may assume that ${\bf x}$ satisfies $x_1\leq 0$ and $x_j\geq 0$ for $j\geq 2$.
This gives us by definition of $\widehat{{\bf x}}$, 
\begin{align*}
    \langle {\bf P}_{G,v},{\bf x}\rangle&= P_{G,v;1}x_1+\sum_{j=2}^{q-1}P_{G,v;j}(x_1+\widehat{x}_j)\\
    &= x_1\sum_{j=1}^{q-1} P_{G,v;j}+\sum_{j=2}^{q-1}P_{G,v;j}\widehat{x}_j\\
    &= -x_1(\pr_{G^{+1},w}[\Phi(v)=q]-\pr_{G^{+q},w}[\Phi(v)=q])+\sum_{j=2}^{q-1}P_{G,v;j}\widehat{x}_j \\
    &=\sum_{j=1}^{q-1} \widehat{P}_{G,v;j}\widehat{x}_j=\langle {\bf \widehat{P}}_{G,v},\widehat{{\bf x}}\rangle,
\end{align*}
where ${\bf \widehat{P}}_{G,v}$ is defined by $\widehat{P}_{G,v;j}=P_{G,v;j}$ for $j\geq 2$ and $\widehat{P}_{G,v;1}=\pr_{G^{+1}}[v=q]-\pr_{G^{+q}}[v=q])$.
We note that by construction and~\eqref{eq:positive}, $\widehat{P}_{G,v;j}\geq 0$ for all $j$. 
Therefore 
\begin{align*}
\langle {\bf \widehat{P}}_{G,v},{\bf \widehat x}\rangle &\leq \|\widehat{{\bf x}}\|_\infty \|{\bf \widehat{P}}_{G,v}\|_1
\\
&=\|\widehat{{\bf x}}\|_\infty |P_{G,v;1}| \leq \|\widehat{{\bf x}}\|_\infty (1-w) \pr_{G^{+q}}[v=1],
\end{align*}
where the last inequality follows in the same way as~\eqref{eq:(1-w) mult bound}.
This finishes the proof.
\end{proof}

\section{Proof of main theorem for \texorpdfstring{$\eta=0$}{eta=0}}\label{sec:proof eta=0}
In this section we prove Theorem~\ref{thm:main expanded} for $\eta=0$.
For this we will use the following lemma, whose proof we defer to the end of this section.
\begin{lemma}\label{lem:lowerbound_wtilde_sum}
Let $\alpha>1$ and let $q,\Delta$ be positive integers such that $q \geq (1+\alpha) \Delta+1$. Then there exist constants $C(\alpha)>0$ and $C_1(\alpha)>0$ such that for any $\varepsilon_1\in (0,C_1(\alpha)/\Delta)$ and $\varepsilon_2\in (0,\pi/8)$  the following holds: for any $(G,v)\in\mathcal{G}_{\Delta,q}^\bullet$ and any $w\in[0,1]$ we have that if $\tilde w\in \mathbb{C}$ and ${\bf x}\in \mathbb{C}^{q-1}$ satisfy 
\begin{align*}
    |\tilde{w}-w| < \varepsilon_1 \quad \text{ and } \quad |x_j - R_{\overline{G},v;j,\ell}(w)| < \varepsilon_2, \text{ for all }j,\ell\in [q] 
\end{align*}
then for any $\tau\in \{0,1\}$ and $\ell\in [q]$,
\begin{align}\label{eq:lower bound on ratio sum}
     \left| \sum_{j\neq \ell} \tilde{w}^{c_j}e^{x_j} + \tilde{w}^{c_\ell+\tau} \right| \geq  C(\alpha)\Delta,
\end{align}
where ${\bf c}={\bf c}_{G,v}$.
\end{lemma}

\begin{proof}[Proof of Theorem~\ref{thm:main expanded} for $\bm{\eta=0}$]
Throughout the proof $q,\Delta$ and $w\in [0,1]$ are fixed and whenever we refer to a graph we in fact mean a partially $q$-colored graph.

We will prove the theorem by induction on the number of free vertices of the graph, the base case being a connected rooted graph $(G,v)\in \mathcal{G}_q^{\bullet}$ where $v$ is the only free vertex of $G$. 
Since the pinned vertices of $G$ form an independent set, $(\overline{G},v)$ is just an isolated vertex and hence for any pair of colors $i,j$ we have
$R_{\overline{G},v;i,j}(\tilde w)=\log(1)=0$ for any $\tilde w \in \mathbb{C}$, proving the first part of the base case.
Now for the second part, we apply Lemma~\ref{lem:lowerbound_wtilde_sum} with $\alpha=2/3$, ${\bf R}={\bf x}={\bf 0}$ and $\tau=0$ to see that provided $\varepsilon_1\leq C_1(2/3)$ and $\varepsilon_2\leq \pi/8$,
\begin{align*}
 |Z_{G}(\tilde w)|=\left|\sum_{j=1}^q \tilde{w}^{c_j}\right|\geq C(2/3)\Delta>0,
\end{align*}
and hence in particular, $Z_{G}(\tilde w)\neq 0$.
This finishes the verification of the base case.

Next consider a partially $q$-colored graph $G$ of maximum degree at most $\Delta$ with more than one free vertex, all of whose pinned vertices are all leaves and form an independent set.
Let $v$ be a free vertex of free degree $d\leq \Delta-1$. 
We need to show $|R_{\overline{G},v;i,j}(\tilde w)-R_{\overline{G},v;i,j}(w)|\leq \varepsilon_2$ and $Z_{G}(\tilde w) \neq 0$. 
By symmetry, we may assume that $i=1, j=q$. 
Note that by induction these log-ratios are well-defined since $Z_{G,v}^i(\tilde w)\neq 0$ for all $i$, because after fixing the color of $v$ in $G$ and replacing $v$ by $d$ many leaves, each connected to a unique neighbor of $v$ in $\overline{G}$, each component of the resulting partially $q$-colored graph $H$ has fewer free vertices than $G$, each original free neighbor of $v$ now has free degree at most $\Delta-1$, its pinned vertices are all leaves and form an independent set. Therefore, by induction and the fact that the partition function is multiplicative over the components of $H$, it follows that $Z_{G,v}^i(\tilde w)=Z_H(\tilde w)\neq 0$.

Choose an ordering of the free neighbors $v_1,\ldots ,v_d$ of $v$  and let $G_i$, $i=1,\ldots, d$ be the graphs obtained from the telescoping procedure. 
Let $I\subseteq [d]$ be the set of indices, where $\pr_{G_i,w}[\Phi(v_i)=1]\le \pr_{G_i,w}[\Phi(v_i)=q]$.   
By~\eqref{eq:difference} we have that
\begin{align}
|R_{\overline{G},v;1,q}(\tilde w)-R_{\overline{G},v;1,q}(w)|\le& \sum_{i=1}^d\underbrace{|F_{w,{\bf c}^i}({\bf R}^i(w))-F_{w,{\bf c}^i}({\bf R}^i(\tilde w))|}_{A_i} \nonumber\\
&+\sum_{i=1}^d\underbrace{|F_{w,{\bf c}^i}({\bf R}^i(\tilde w))-F_{\tilde w,{\bf c}^i}({\bf R}^i(\tilde w))|}_{B_i}, \label{eq:bound on diff}
\end{align}
where the vectors ${\bf R}^i(\tilde w)\in \mathbb{C}^{q-1}$ are defined by $R^i_{j}(\tilde w)=R_{\overline{G}_i,v_i;j,q}({\tilde w})$ for $j=1,\ldots,q-1$ (by induction these are well-defined, since the graphs $\overline{G}_i$ have fewer free vertices), and where ${\bf c}^i={\bf c}_{G,v_i}$. 

First let us bound the first summation.  By \eqref{eq:difference in terms of gradient} for a given $i\in [d]$, we know that 
\begin{align*}
A_i=|F_{w,{\bf c}^i}({\bf R}^i(w))&-F_{w,{\bf c}^i}({\bf R}^i(\tilde w))|\\ 
&\le \sup_{t\in[0,1]}|\langle \nabla F_{w,{\bf c}^i}(t{\bf R}^i(w)+(1-t){\bf R}^i(\tilde w)),{\bf R}^i(w)-{\bf R}^i(\tilde w)\rangle|\\
&=\sup_{t\in[0,1]}|\langle \nabla F_{w,{\bf c}^i}({{\bf y}_t}),{\bf x}\rangle|,
\end{align*}
where ${\bf x}={\bf R}^i(w)-{\bf R}^i(\tilde w)$ and ${{\bf y}_t}=t{\bf R}^i(w)+(1-t){\bf R}^i(\tilde w)$. By induction we know that 
\begin{align*}
|R_{\overline{G}_i,v_i;\ell_1,\ell_2}(\tilde w)-R_{\overline{G}_i,v_i;\ell_1,\ell_2}(w)|\leq \varepsilon_2 
\end{align*} 
for any pair of colors $\ell_1,\ell_2$. In particular, taking $\ell_2=q$ we have $\|{\bf x}\|_\infty\le \varepsilon_2 $ and taking $\ell_2=1$ we have $\|\widehat{{\bf x}}\|_\infty\le \varepsilon_2$, where $\widehat{\bf x}$ is defined as $\widehat{x}_1:=-x_1$ and $\widehat{x}_j:=x_j-x_1$ for $j=2,\dots,q-1$.
Moreover, for any $t\in[0,1]$ the vector ${\bf y}_t=t{\bf R}^i(w)+(1-t){\bf R}^i(\tilde w)$ satisfies
\[
\|{\bf y}_t\|_\infty\le \varepsilon_2 \quad\textrm{and}\quad \|{\bf y}_t-{\bf R}^i(w)\|_\infty=(1-t)\|{\bf x}\|_\infty\le \varepsilon_2 .
\]
Then we can bound $A_i$ as follows:
\begin{align}
A_i&\le \sup_{t\in[0,1]}\left|\langle {\bf P}_{{G}_i,v_i},{\bf x} \rangle+\langle \nabla F_{w,{\bf c}^i}({\bf y}_t)-{\bf P}_{{G}_i,v_i},{\bf x} \rangle \right|\nonumber
\\
&\le \sup_{t\in[0,1]}\left(|\langle {\bf P}_{{G}_i,v_i},{\bf x}\rangle|+\|{\bf P}_{{G}_i,v_i}-\nabla F_{w,{\bf c}^i}({{\bf y}_t})\|_1\|{\bf x}\|_\infty\right)\nonumber
\\
&\le |\langle {\bf P}_{{G}_i,v_i},{\bf x}\rangle|  +   \varepsilon_2\sup_{t\in[0,1]}\|{\bf P}_{{G}_i,v_i}-\nabla F_{w,{\bf c}^i}({{\bf y}_t})\|_1. \label{eq:bounding A_i}
\end{align}

By Proposition~\ref{prop:gradient and w}, if $\varepsilon_2>0$ is sufficiently small (i.e. $\varepsilon_2<\delta$ that is given by applying  Proposition~\ref{prop:gradient and w}(i) with $\varepsilon=1/(3\Delta^2)$), then 
$\|{\bf P}_{{G}_i,v_i}-\nabla F_{w,{\bf c}^i}({{\bf y}_t})\|_1\le 1/(3\Delta^2)$.  To bound $|\langle {\bf P}_{{G}_i,v_i},{\bf x}\rangle|$ we distinguish two cases depending on whether $i\in I$ or $i\in [d]\setminus I$. 
If $i\in I$, then by Lemma~\ref{lem:replace by prob} and Lemma~\ref{lem:prob basic}  we can further bound $A_i$ as
\begin{align}
A_i&\le (1-w)\pr_{G_i^{+1}}[v_i=q] \cdot \|{\bf x}\|_{\infty} +\varepsilon_2\tfrac{1}{3\Delta^2} \label{eq:bound Ai}
\\
 &\le  \varepsilon_2\left(\frac{(1-w)}{q-\Delta}+\frac{1}{3\Delta^2}\right).\nonumber
\end{align}
Similarly, if $i\in [d]\setminus I$ we can argue the same way, implying that for any $i=1,\ldots,d$ we have
\begin{align}\label{eq:crucial_for_2Delta}
A_i\le  \varepsilon_2\left(\frac{(1-w)}{q-\Delta}+\frac{1}{3\Delta^2}\right).
\end{align}

Now let us bound $B_i$ in~\eqref{eq:bound on diff}. By Proposition~\ref{prop:gradient and w} if $\varepsilon_1>0$ is sufficiently small (i.e. $\varepsilon_1<\delta$ that is given by applying Proposition~\ref{prop:gradient and w}(ii) with $\varepsilon=\varepsilon_2/(3\Delta^2)$) we have
\begin{align}
B_i\le \varepsilon_2/(3\Delta^2). \label{eq:bound on B_i}
\end{align}
By substituting \eqref{eq:crucial_for_2Delta} and \eqref{eq:bound on B_i} into \eqref{eq:bound on diff} we obtain
\begin{align*}
|R_{\overline{G},v;1,q}(\tilde w)-R_{\overline{G},v;1,q}(w)|&\le d\varepsilon_2\left(\frac{1-w}{q-\Delta}+\frac{2}{3\Delta^2}\right)
\\
&\le \varepsilon_2\left(\frac{\Delta-1}{\Delta}+\frac{2}{3\Delta}\right)<\varepsilon_2.
\end{align*}
This proves the first part of the statement.

To show that $Z_G(\tilde w)\neq 0$ we argue as follows.
Choose any color $\ell$ that is not blocked at the vertex $v$. 
Then $Z_{G,v}^\ell(\tilde w)=Z_{\overline{G},v}^\ell(\tilde w)$, since color $\ell$ is not blocked at $v$.
Therefore, by induction we have that $Z_{G,v}^\ell(\tilde w)\neq 0$, since after replacing $v$ by $\deg_{\overline{G}}(v)$ many leaves each connected to a unique neighbor of $v$ in $\overline{G}$ each component of the resulting partially $q$-colored graph is contained in $\mathcal{G}_{\Delta,q}^{\bullet}$ and has fewer free vertices.

Therefore, to prove that $Z_{G}(\tilde w)\neq 0$, it is sufficient to prove that 
\begin{align*}
    0\neq \frac{Z_{G}(\tilde w)}{Z_{G,v}^\ell(\tilde w)}=\sum_{j=1}^q {\tilde w}^{c_j}e^{R_{\overline G,v;j,\ell}(\tilde w)}.
\end{align*}
This follows directly from an application of Lemma~\ref{lem:lowerbound_wtilde_sum} with $x_j=R_{\overline G,v;j,\ell}(\tilde w)$ and $\alpha=2/3$.
\end{proof}

\begin{remark}\label{rem:choice of eps1}
Note that our proof shows that we need $\varepsilon_2$ to be smaller than $\delta$ with $\delta$ coming from Proposition~\ref{prop:gradient and w} with $\epsilon=1/(3\Delta^2)$.
Additionally, we need $\varepsilon_1$ to be smaller than $\delta$ coming from Proposition~\ref{prop:gradient and w} with $\epsilon=\varepsilon_2/(3\Delta^2)$.

In the appendix we give an alternative proof of Proposition~\ref{prop:gradient and w} that shows that we can take $\delta\geq  C\varepsilon$ for a constant $C>0$ for these choices of $\varepsilon$. 
Therefore we conclude that $\varepsilon_1$ can be chosen to be $C'\Delta^{-4}$ for some constant $C'$.
\end{remark}
\begin{remark}\label{rem:marginal} 
Note that the proof actually shows that we can inductively bound the log-ratio difference by $\varepsilon_2 d/\Delta$, where $d$ denotes the free degree of the vertex $v$, as the sum in~\eqref{eq:ratio difference for induction} is over the free neighbors of $v$.
The crucial part in the proof is to bound $A_i$. 
We do this in~\eqref{eq:bound Ai} using the marginal probability of the root vertex being colored with color $1$, which we bound by $\tfrac{1-w}{q-\Delta}\leq \tfrac{1-w}{\Delta}$ in~\eqref{eq:crucial_for_2Delta} under the assumption that $q\geq 2\Delta$.
It is not hard to see that if for a restricted family of graphs of maximum degree at most $\Delta$ this marginal probability can be bounded by $\tfrac{1}{d+c}$ for some $c\in (0,1)$ (where $d$ denotes the free degree) under weaker assumptions on $q$, then the entire proof still applies under these weaker assumptions, but with the modified induction assumption described above. 
This way we essentially recover the condition from Liu, Sinclair and Srivastava~\cite{LSS2Delta} (as stated in $(\star)$ in the introduction.) In fact, our condition is slightly weaker.
In particular for triangle free graphs it is known that such bounds exist under the assumption that $q\geq 1.7633 \Delta+\beta$ for some absolute constant $\beta>0$, see~\cite{GKMssm,LSS2Delta} for details.

The next section contains refined bounds on the marginal probability of the root vertex under additional assumptions on the structure of the neighborhood of that vertex that we utilize in the proof of the main theorem for $\eta>0$.
\end{remark}
We end this section with a proof of Lemma~\ref{lem:lowerbound_wtilde_sum}. 
In the proof we will use the following lemma of Barvinok.
\begin{lemma}[\protect{Barvinok~\autocite[Lemma 3.6.3]{Barbook} }] \label{Lem:Barvinok}
    Let $u_1, \ldots , u_k \in \R^2$ be non-zero vectors such that the angle between any vectors $u_i$ and $u_j$ is at most $\varphi$ for some $\varphi \in [0, 2\pi/3)$. 
    Then the $u_i$ all lie in a cone of angle at most $\varphi$ and
    \[
    \left| \sum_{j=1}^k u_j \right| \ge \cos(\varphi/2)\sum_{j=1}^k|u_j|.
    \]
\end{lemma}

\begin{proof}[Proof of Lemma~\ref{lem:lowerbound_wtilde_sum}]
For the proof we may without loss of generality assume that $\ell=q$.

    Let $ w^\star =  \varepsilon_1 /\sin\left( \frac{\pi}{8 \Delta} \right)$.
    We distinguish the cases $w \leq w^\star$ and $w > w^\star $.

    First consider the case $w \leq  w^\star $.      
By Lemma~\ref{lem:prob basic} and Lemma~\ref{lem:prob basic lower}, and by noting that $q/(q-\Delta)$ is decreasing in $q$ for $q>\Delta$, we have that 
 \begin{equation*}
 \frac{\alpha \Delta +1}{\left((1+\alpha)\Delta +1\right) e^{1/\alpha}} \leq   \frac{q-\Delta}{qe^{1/\alpha}} \leq \exp(R_j)\le  \frac{qe^{1/\alpha}}{q-\Delta} \leq \frac{ \left((1+\alpha)\Delta +1\right) e^{1/\alpha} }{\alpha \Delta +1} ,
 \end{equation*}
and therefore in particular 

 \begin{equation}
     \frac{ \alpha}{(1+\alpha)e^{1/\alpha}} \le \exp(R_j)\le  \frac{(1+ \alpha)e^{1/\alpha}}{\alpha}.
\end{equation}
Thus, by assumption we have $M^{-1}e^{-\varepsilon_2} \leq e^{x_j}\leq Me^{\varepsilon_2}$, where $M=  \frac{(1+ \alpha)e^{1/\alpha}}{\alpha}$, and that all $e^{x_j}$ are contained in a cone of angle at most $2\varepsilon_2\leq \pi/4$, centered around the real axis.
    Moreover, we remark that at least $q- \Delta $ entries of $\bm{c}$ are zero (including possibly $c_q$). Let $S_0 \subset[q-1]$ denote the indices for which this is the case. 
   We obtain
    \begin{align*}
        \left| \sum_{j=1} ^{q-1} \tilde{w}^{c_j}e^{x_j} + \tilde{w}^{c_q+\tau} \right| &\ge  \left| \sum_{j \in S_0} \tilde{w}^{c_j}e^{x_j}  \right|  - \left| \sum_{j \in [q-1]\setminus{S_0}} \tilde{w}^{c_j}e^{x_j}  + \tilde{w}^{c_q+\tau} \right| \\
        &\ge \cos(\varepsilon_2)\sum_{j \in S_0} | e^{x_j}| - \sum_{j \in [q-1]\setminus{S_0}} |\tilde{w}|^{c_j} |e^{x_j}|  - |\tilde{w}^{c_q+\tau}|\\
        &\ge \cos(\varepsilon_2) \frac{(q-\Delta-1)}{M e^{\varepsilon_2}} - \Delta e^{\varepsilon_2} M ( \varepsilon_1 + w^\star)\\
        & \geq \frac{(q-\Delta-1)}{2M} - 2\Delta M  ( \varepsilon_1 +w^\star ).
    \end{align*}
In the second step we used Lemma~\ref{Lem:Barvinok}, and in the last step we used the fact that $\varepsilon_2 \le \pi/8 < \frac{1}{2}$ thus $\cos(\varepsilon_2)e^{-\varepsilon_2}\geq 1/2$.

    For the case $w>w^\star$, we find that
    \begin{align*}
        |\arg( \tilde{w}^{c_j+\tau}e^{x_j})| \leq (c_j+\tau)|\arg{\tilde{w}}| + \varepsilon_2 \leq (c_j+\tau) \arcsin\left( \frac{\varepsilon_1}{w^\star} \right) + \varepsilon_2  \leq \Delta \frac{\pi}{8\Delta} + \frac{\pi}{8} \leq \frac{\pi}{4}.
    \end{align*}

Using Lemma~\ref{Lem:Barvinok} again, we obtain
\begin{align*}
    \left| \sum_{j=1}^{q-1} \tilde{w}^{c_j}e^{x_j} + \tilde{w}^{c_q+\tau} \right| &\ge\cos\left( \frac{\pi}{4}\right) \sum_{j=1} ^{q-1} \left|  \tilde{w}^{c_j}e^{x_j} \right| + \cos\left( \frac{\pi}{4}\right) \left|\tilde{w}^{c_q+\tau} \right|\\
    &\geq 
    \cos\left(\frac{\pi}{4} \right)\sum_{j \in S_0}\left|  \tilde{w}^{c_j}e^{x_j}  \right| \\
    &\ge \frac{(q-\Delta-1)}{M} \frac{\cos\left( \frac{\pi}{4}\right) }{e^{\varepsilon_2}}.
    \\
    &\ge \frac{(q-\Delta-1)}{3M}.
\end{align*}
Combining the two cases, we may conclude that

\[
 \left| \sum_{j=1} ^{q-1} \tilde{w}^{c_j}e^{x_j} + \tilde{w}^{c_q+\tau} \right| \geq  \frac{(q-\Delta-1)}{3M} - 2\Delta M \varepsilon_1 \left(1 + \frac{1}{\sin\left(  \frac{\pi}{8\Delta }\right)}\right).
\]

Let us further bound this by observing that
\[
\frac{q-\Delta-1}{M}  \geq \frac{\alpha^2 }{(1+\alpha)e^{1/\alpha}} \Delta,
\]
and
\begin{align*}
  1 + \frac{1}{\sin\left(  \frac{\pi}{8\Delta} \right)}  & = \frac{\sin\left(  \frac{\pi}{8\Delta} \right)+1}{\sin\left(  \frac{\pi}{8\Delta } \right)} \leq \frac{2}{\sin\left(  \frac{\pi}{8\Delta } \right)}  \leq \frac{2}{  \left(  \frac{\pi}{8\Delta} \right)/2   } = \frac{32 \Delta}{\pi} \leq 11 \Delta.
\end{align*}
 This yields us 
 \begin{align*}
      \left| \sum_{j=1} ^{q-1} \tilde{w}^{c_j}e^{x_j} + \tilde{w}^{c_q+\tau} \right| \geq \left(\underbrace{\frac{\alpha^2 }{3(1+\alpha)e^{1/\alpha}}}_{f(\alpha)}  - \underbrace{22 \frac{e^{1/\alpha} (1+ \alpha)}{\alpha}}_{g(\alpha)} \varepsilon_1 \Delta \right) \Delta .
 \end{align*}
Rewriting the right-hand side as $(f(\alpha) -g(\alpha) \varepsilon_1 \Delta)\Delta$, we see that if $0\le \varepsilon_1\le \frac{f(\alpha)}{2g(\alpha)\Delta}$, then
\[
(f(\alpha) -g(\alpha) \varepsilon_1 \Delta)\Delta\geq \tfrac{1}{2}f(\alpha)\Delta.
\]
This means that the choice $C(\alpha)=\tfrac{1}{2}f(\alpha)$ and $C_1(\alpha)=\frac{f(\alpha)}{2g(\alpha)}$ satisfies the condition of the lemma.
\end{proof}

\section{Refined bounds on marginal probabilities of the root vertex} \label{sec:more marginal bounds}
We collect here some results that improve on the bounds on the marginal probability of the root vertex from Lemma~\ref{lem:prob basic}, under additional assumptions on the local structure of the neighborhood of the root vertex that we will utilize in the proof of the main theorem for $\eta>0$ in the next section.
For the proof of Theorem~\ref{thm:main expanded} one only needs the conclusions of Corollaries~\ref{cor: few blocked} and~\ref{cor:sparse neighborhood}. Apart from that, this section can be read independently of the next section.

\subsection{Local influences on the probabilities}
Our first lemma shows how to improve the bound from Lemma~\ref{lem:prob basic} under the additional assumption that not all neighbors of the root vertex have color $j$ blocked.

\begin{lemma}\label{lem: few blocked}
Let $q$ and $\Delta$ be positive integers such that $q\geq (1+\alpha)\Delta+1$ for some $\alpha>0$ and let $G$ be a partially $q$-colored graph of maximum degree at most $\Delta$ and let $v\in V(G)$ be a free vertex of degree $d$ and with free-degree $f$.
Let $\gamma\in [0,1]$ and let $j\in \{1,\ldots,q\}$ be a free color at $v$.
Suppose that color $j$ is blocked for at most $(1-\gamma)f$ of the free neighbors of $v$  
Then for any $w\in [0,1]$,
\[
\pr_{G,w}[\Phi(v)=j]\leq \frac{(1-w)\exp\left(\tfrac{-\gamma f}{qe^{1/\alpha}}\right)+w}{q-d+dw}.
\]

\end{lemma}
\begin{proof}
Throughout we will fix $w$ and just write $\pr_G$ instead of $\pr_{G,w}$.
By~\eqref{eq:expand prob} we have     
\begin{align*}
    \pr_{G}[\Phi(v)=j]=\sum_{\kappa}\pr_G[\Phi(v)=j\mid E_\kappa]\pr_{G}[E_\kappa],
\end{align*}
where $E_\kappa$ denotes the event that the random coloring $\Phi$ agrees with $\kappa$ on $N(v)$, and we sum over all colorings $\kappa$ of the neighbors of $v$ for which $\pr_{G}[E_\kappa]\neq 0$.
By~\eqref{eq:prob E_kappa}, for a given $\kappa$ that does not use color $j$, we can bound $\pr_G[v=j\mid E_\kappa]$ by $\tfrac{1}{q-d+dw}$, and if $\kappa$ does use color $j$, then we bound this probability by $\tfrac{w}{q-d +dw}$.
So we can bound $\pr_G[\Phi(v)=j]$ by 
\begin{align*}
    \frac{1-w}{q-d+dw}\pr_{G}[j\notin \Phi(N(v))] +\frac{w}{q-d+dw}.
\end{align*}

Our next step is to bound $\pr_{G}[j\notin \Phi(N(v))]$.
Let $u_1,\ldots, u_t$ be the free neighbors of $v$ for which color $j$ is not blocked.
Note that $t\geq \gamma f.$
Then as in~\eqref{eq:product expand}, we can write
\begin{align}
\pr_{G}[j\notin\Phi(N(v) ]\leq \prob_G[j\notin \Phi\{u_1,\ldots,u_t\}]=\prod_{i=1}^t \pr_{G}[\Phi(u_i)\neq j\mid E_i],\label{eq:bound on not using 1}
\end{align}
where $E_i$ denotes the event that the random coloring $\Phi$ does not use color $j$ on the vertices $u_{i+1},\ldots,u_t$.

We next claim that for each $i=1,\ldots, t$ we have the following the lower bound
\begin{align}\label{eq:bound on u_i=1}
\pr_{G}[\Phi(u_i)=j \mid E_i]\geq \frac{1}{qe^{1/\alpha}}.
\end{align}
We prove~\eqref{eq:bound on u_i=1} below, but first we plug this into~\eqref{eq:bound on not using 1} to obtain
\begin{align*}
\pr_{G}[j\notin\Phi(N(v) ]&\le  \left(1-\frac{1}{q e^{1/\alpha}}\right)^{t}\le  \left(1-\frac{1}{q e^{1/\alpha}}\right)^{\gamma f}
\leq \exp\left(\tfrac{-\gamma f}{qe^{1/\alpha}}\right).
\end{align*}

It thus remains to show~\eqref{eq:bound on u_i=1}.
If it were not for the event $E_i$, this would be a direct consequence of Lemma~\ref{lem:prob basic lower}. 
We follow the proof of that lemma to show that the same bound still applies.
We again expand 
\begin{align*}
    \pr_{G}[\Phi(u_i)=j \mid E_i]=\sum_{\kappa}\pr_G[\Phi(v)=j\mid E_\kappa \cap E_i]\pr_{G}[E_\kappa\mid E_i],
\end{align*}
where the sum is over $\kappa$ for which $\Pr_G[E_\kappa\cap E_i]>0$.
By~\eqref{eq:prob E_kappa}, we can lower bound $\pr_G[\Phi(v)=j\mid E_\kappa \cap E_i]$ for any such $\kappa$ that additionally does not assign color $j$ to the neighbors of $u_i$ by $1/q$.
It thus suffices to lower bound $ \pr_{G}[j\notin \Phi(N(u_i))\mid E_i]$ by $e^{-1/\alpha}$.
To this end denote by $v_1,\ldots,v_d$ the neighbors of $u_i$, and note that by assumption none of them is pinned to color $j$. 
Then as in~\eqref{eq:product expand} we have
\begin{align*}
    \pr_{G}[j \notin \Phi(N(v))\mid E_i] &= \prod_{k=1}^d  \pr_{G}[ \Phi(v_k) \neq j \mid \hat{E}_k] \nonumber
    \\
    &=  \prod_{k=1}^d \left( 1- \pr_{G}[ \Phi(v_k) = j \mid \hat{E}_k] \right),
\end{align*}
where $\hat{E}_k$ denotes the event that $\Phi$ does not use color $j$ on $v_{k+1},\ldots,v_d$ and not on $u_{i+1},\ldots, u_{t}$.   
Now since $q>\Delta+1$, we can again use the Markov property (cf.~\eqref{eq:prob E_kappa}) to bound  $\pr_{G}[ \Phi(v_k) = j \mid \hat{E}_k]$ by $\tfrac{1}{q-\Delta+\Delta w}$.
It then follows as in~\eqref{eq:lower bound  on not using 1} that $ \pr_{G}[j\notin \Phi(N(u_i))\mid E_i]$ is at least $e^{-1/\alpha}$ and therefore~\eqref{eq:bound on u_i=1} holds. 
\end{proof}

We record here some concrete applications of the previous lemma that will be used in the proof of Theorem~\ref{thm:main expanded} in the next section.
\begin{corollary}\label{cor: few blocked}
  Let $q$ and $\Delta\geq 500$ be positive integers such that $q\geq (2-\eta)\Delta$ for some $0\leq\eta\le 0.002$, let $G$ be a partially $q$-colored graph of maximum degree $\Delta$, and let $v\in V(G)$ be a free vertex of degree $d$ and with free-degree $f\le \Delta-1$.
Let $\gamma\in [0,1]$ and let $j\in \{1,\ldots,q\}$ be a free color at $v$.
Suppose that color $j$ is blocked for at most $(1-\gamma)f$ of the neighbors of $v$.
Then, if $w\in [0,1]$ and 
\begin{itemize}
   \item $\gamma\ge 0.02$, then $\pr_{G,w}[\Phi(v)=j]\leq \tfrac{1}{f+1}$,
   \item $\gamma\ge 0.14$, then $\pr_{G,w}[\Phi(v)=j]\leq \tfrac{0.977}{f+1}$.
\end{itemize}
\end{corollary}
\begin{proof} 
Let $d$ be the degree of $v$ and apply Lemma~\ref{lem: few blocked} with $\alpha=1-\eta-1/\Delta\geq 0.996$. Now the bound on $\pr_{G,w}[\Phi(v)=j]$ from Lemma~\ref{lem: few blocked} is clearly decreasing in terms of $\alpha$,  $w$ and increasing in terms of $d$. 
Therefore, we may assume that $w=0$, $\alpha=0.996$ and $d=\Delta$. Then
\begin{align*}
\pr_{G}[\Phi(v)=j]\leq \frac{\exp\left(\tfrac{-\gamma f}{qe^{1000/996}}\right)}{q-\Delta}.
\end{align*}
Let us denote $g:[0,\Delta-1]\to \mathbb{R}$ to be the function defined by $g(f)=(f+1)\exp\left(-\frac{\gamma f}{qe^{1000/996}}\right)$. This function is increasing on $[0,\Delta-1]$, therefore we have
\[
(f+1)\pr_{G}[\Phi(v)=j]\leq \frac{\Delta\exp\left(\tfrac{-\gamma (\Delta-1)}{qe^{1000/996}}\right)}{q-\Delta},
\]
that is decreasing in $q$ for $\gamma\in[0,1]$. Thus choosing $q=1.998\Delta$ we have
\[
(f+1)\pr_{G}[\Phi(v)=j]\leq \frac{\Delta\exp\left(\tfrac{-\gamma (1-1/\Delta)}{1.998e^{1000/996}}\right)}{0.998}\le \frac{\exp\left(\tfrac{-0.998 \gamma }{1.998e^{1000/996}}\right)}{0.998} .
\]
As this bound is decreasing in $\gamma$, by choosing $\gamma=0.02$ and $\gamma=0.14$ we obtain the desired statement, as can be numerically verified.
\end{proof}

The next lemma is well known, variations have for example been used in~\cite{GKMssm,Luetalfptasforcubic}. We give a proof for completeness.
\begin{lemma}\label{lem:eion trick}
Let $q$ be a positive integer.
Let $G$ be partially $q$-colored graph of maximum degree $\Delta\leq q-1$ and let $v\in V(G)$ be a free vertex. 
Then for any $w\geq 0$,
\begin{equation}\label{eq:eoin trick}
      \pr_{G,w}[\Phi(v)=q]=\frac{\mathbb{E}_{G-v,w}[w^{c_q\left(\Phi|_{N(v)}\right)}]}{\sum_{i=1}^{q}\mathbb{E}_{G-v,w}[w^{c_i\left(\Phi|_{N(v)}\right)}]}.
\end{equation}
  
\end{lemma}
\begin{proof}
Since $q\geq \Delta+1$, we know that $G-v$ can be properly colored, i.e. $Z_{G-v}(w)\neq 0$. Then we have the following identity:
\[
    \frac{Z^j_{G,v}(w)}{Z_{G-v}(w)}=\sum_{\substack{\psi:V\setminus\{v\}\to  q\\ \psi|_S=\phi}} \frac{w^{c_j\left(\psi|_{N(v)}\right)}w^{m(\psi)}}{Z_{G-v}(w)}=\mathbb{E}_{G-v,w}[w^{c_j\left(\Phi|_{N(v)}\right)}].
\]

Now the claim follows from this identity:
\begin{align*}
\pr_{G,w}[\Phi(v)=q]&=\frac{Z_{G,v}^q(w)}{Z_G(w)}
=\frac{Z_{G,v}^q(w)}{\sum_{j=1}^qZ_{G,v}^j(w)}\\
&=\frac{Z_{G,v}^q(w)/Z_{G-v}(w)}{\sum_{j=1}^qZ_{G,v}^j(w)/Z_{G-v}(w)}
=\frac{\mathbb{E}_{G-v,w}[w^{c_q\left(\Phi|_{N(v)}\right)}]}{\sum_{i=1}^{q}\mathbb{E}_{G-v,w}[w^{c_i\left(\Phi|_{N(v)}\right)}]}.
\end{align*}
\end{proof}

Our next lemma shows how to improve the bound from Lemma~\ref{lem:prob basic} under the additional assumption that the graph induced by the neighbors of the root vertex is not `close to' a clique.

\begin{lemma}\label{lem:sparse neighborhood}
Let $q$ and $\Delta$ be positive integers such that $q\geq (1+\alpha)\Delta+1$ for some $\alpha>0$ and let $G$ be a partially $q$-colored graph of maximum degree $\Delta$ such that the pinned vertices of $G$ are all leaves and let $v\in V(G)$ be a vertex of free degree $f$. 
Let $H$ be the graph induced by $N_G(v)$ and let $L$ be the collection of free colors at $v$. 
Then for any $j\in L,$ and $w\in [0,1],$
\[
\pr_{G,w}[\Phi(v)=j]\le 
\frac{1}
{|L|\left(1-\frac{1-w}{q-\Delta+1-w}\right)^{\left((q-\Delta)f+2e(H)+f\right)/|L|}\left(1-\frac{w}{q-\Delta+1}\right)^{fq/|L|}}.
\]
\end{lemma}
\begin{proof}
In the proof we will fix $w$ and just write $\pr_G$ instead of $\pr_{G,w}$.
We will apply Lemma~\ref{lem:eion trick}. 
We can just upper bound the numerator in~\eqref{eq:eoin trick} by $1$.
We thus need a lower bound on \[\sum_{\ell=1}^{q}\mathbb{E}_{G-v}[w^{c_\ell\left(\Phi|_{N(v)}\right)}].\] 
Let ${\bf c}={\bf c}_{G,v}$ and let us denote the free neighbors of $v$ by $v_1,\ldots,v_f$.
We have
    \begin{align*}
    \sum_{\ell=1}^{q}\mathbb{E}_{G-v}[w^{c_\ell\left(\Phi|_{N(v)}\right)}]&\ge \sum_{\ell\in L}\mathbb{E}_{G-v}[w^{c_\ell\left(\Phi|_{N(v)}\right)}],
    \end{align*}
    which we can lower bound by AM-GM by
    \begin{align*}
    |L|\prod_{\ell\in L}\mathbb{E}_{G-v}[w^{c_\ell\left(\Phi|_{N(v)}\right)}]^{1/|L|}& \ge |L|\prod_{\ell\in L}\pr_{G-v}[\ell\notin\Phi(N)]^{1/|L|}\\
    &=|L| \prod_{\ell\in L} \prod_{i=1}^f \left(1- \pr_{G-v}[ \Phi(v_i)=\ell \mid E_i^\ell ]\right)^{1/|L|}\\
    &= |L| \prod_{i=1}^f  \prod_{\ell\in L}  \left(1- \pr_{G-v}[ \Phi(v_i)=\ell \mid E_i^\ell ]\right)^{1/|L|},
\end{align*}
where $E_i^\ell$ is the event that $v_{i+1}, \ldots, v_f$ do not use color $\ell$.
Let $Q_i \subseteq [q]$ be the set of colors appearing in the neighborhood of $v_i$ (in $G$), and let $d_H(v_i)$ denote $\deg_H(v_i)$. 
Note that $|Q_i| + d_H(v_i) \le \Delta-1$, since the pinned neighbors of $v$ are isolated vertices of $H$.
By Lemma~\ref{lem:prob basic} if $\ell\notin Q_i$, then
\[
    \pr_{G-v}[ \Phi(v_i)=\ell \mid E^\ell_i ]   \leq \frac{1}{q - |Q_i| - d_H(v_i)}
\]
and if $\ell\in Q_i$, then
\[
    \pr_{G-v}[ \Phi(v_i)=\ell \mid E^\ell_i]   \leq \frac{w}{q - |Q_i| - d_H(v_i)}.
\]
Hence, for a fixed $i$, we obtain that $\prod_{\ell \in L}  \left(1- \pr_{G-v}[ \Phi(v_i)=\ell \mid E^\ell_i ]\right)$ is at least
\begin{align}\label{eq:prod bound}
\left(1- \frac{1}{q - |Q_i| - d_H(v_i)}\right)^{|L\setminus Q_i|}\left(1- \frac{w}{q - |Q_i| - d_H(v_i)}\right)^{|L\cap Q_i|}.
\end{align}
We next show how to bound both factors in~\eqref{eq:prod bound}.

Since $q-|Q_i|\ge |L\setminus Q_i|$ and $q-|Q_i|\ge q-(\Delta-1)+d_H(v_i)$ we have
\begin{align*}
     &\left(1- \frac{1}{q - |Q_i| - d_H(v_i)}\right)^{|L\setminus Q_i|} \geq
     \left(1- \frac{1}{q - |Q_i| - d_H(v_i)}\right)^{q - |Q_i|}  \\ 
     =& \left(1- \frac{1}{q - |Q_i| - d_H(v_i)}\right)^{\frac{(q-|Q_i|-d_H(v_i))}{q-(\Delta-1)}\frac{(q - |Q_i|)(q-(\Delta-1))}{(q-|Q_i|-d_H(v_i))}}
     \\
     \geq & \left(1- \frac{1}{q - (\Delta-1)}\right)^{\frac{(q - |Q_i|)(q-(\Delta-1)}{q-|Q_i|-d_H(v_i)}}
     \geq \left(1- \frac{1}{q - (\Delta-1)}\right)^{q-\Delta+d_H(v_i)+1}.
\end{align*}
where in the second inequality we use Bernoulli inequality, and in the last inequality we used the fact that $x\mapsto \frac{x}{x-d_H(v_i)}$ is decreasing in $x$ for $x>d_H(v_i)$.

Since the function $x\mapsto (1-\tfrac{w}{q-x-d_{H}(v_i)})^x$ is monotonically decreasing in $x$ for $0\le x <q- d_{H}(v_i)-w$, we have
\[
\left(1-\frac{w}{q-|L\cap Q_i|-d_H(v_i)}\right)^{|L\cap Q_i|}\ge \left(1-\frac{w}{q-\Delta+1}\right)^{\Delta-1-d_H(v_i)}.
\]
Together this implies that $\prod_{\ell \in L}  \left(1- \pr_{G-v}[ \Phi(v_i)=\ell \mid E_i^\ell ]\right)$ is lower bounded by
\begin{align*}
&\left(1- \frac{1}{q - \Delta+1}\right)^{q-\Delta+1+d_H(v_i)}\left(1- \frac{w}{q - \Delta+1}\right)^{\Delta-1-d_H(v_i)}
\\
=& \left(1- \frac{1}{q - \Delta+1}\right)^{q-\Delta+1+d_H(v_i)} \left(1- \frac{w}{q - \Delta+1}\right)^{(\Delta-1-d_H(v_i) -q)+q} 
\\
= 
&\left(1-\frac{1-w}{q-\Delta+1-w}\right)^{q-\Delta+1+d_H(v_i)}\left(1-\frac{w}{q-\Delta+1}\right)^q.
\end{align*}
Taking the product over the free vertices and realizing that $2e(H)$ is the sum of the degree of the vertices in $H$, as $d_H(v_i)=0$ for a pinned vertex $v_i$, gives us the desired lower bound for $\sum_{\ell=1}^{q}\mathbb{E}_{G-v}[w^{c_\ell(\Phi|_{N(v)})}]$ and this finishes the proof.
\end{proof}

We record as a corollary of this lemma a version with explicit bounds on some of the parameters involved, making it easier to apply in our proof of Theorem~\ref{thm:main expanded}.

\begin{corollary}\label{cor:sparse neighborhood}
Let $q$ and $\Delta$ be positive integers such that $q\geq (2-\eta)\Delta$ for some $0\leq \eta\leq 0.002$ and $\Delta\ge 500$. Let $G$ be a partially $q$-colored graph of maximum degree $\Delta$ whose pinned vertices are all leaves and let $v\in V(G)$ be a vertex of free degree $f\geq (1-\eta)\Delta-2/3$. 
Let $H$ be the graph induced by $N_G(v)$ and let $L$ be the collection of free colors at $v$ and assume that $H$ has average degree $\overline{d} \le 0.36f$.
Then for any $j\in L,$ and $w\in [0,0.002],$
\[
\pr_{G,w}[\Phi(v)=j]\le \tfrac{1}{f+1}.
\]
\end{corollary}
\begin{proof}
First note that the number of free colors $|L|\ge q-\Delta+f$, thus $\frac{f}{|L|}\le\frac{f+1}{|L|}\le  \frac{\Delta}{q-1}\le \frac{1}{2-\eta-1/\Delta}$ and 
\[
    \frac{(q-\Delta+\overline{d}+1)f}{|L|}\le \frac{q-\Delta+0.36 \Delta+1}{2-\eta-1/\Delta}.
\]
Now let us apply Lemma~\ref{lem:sparse neighborhood} and use the previous bounds to bound the exponents in the denominator. Thus, we obtain
\[
\pr_{G,w}[\Phi(v)=j]\leq \frac{1}
{|L| \left(1-\frac{1-w}{q-\Delta+1-w}\right)^{\left(q-\Delta+0.36\Delta+1\right)/(2-\eta-1/\Delta)}\left(1-\frac{w}{q- \Delta+1}\right)^{q/(2-\eta-1/\Delta)}}.
\]
Now let us bound the functions appearing in the denominator. Using the Bernoulli inequality we have
\[
\left(1-\frac{1-w}{q-\Delta+1-w}\right)^{\left(q-\Delta+0.36\Delta+1\right)/(2-\eta-1/\Delta)}\ge  \left(1-\frac{1-w}{(1-\eta)\Delta}\right)^{\frac{\left(q-\Delta+0.36\Delta+1\right)(1-\eta)\Delta}{\left(q-\Delta+1-w\right)(2-\eta-1/\Delta)}},
\]
where the exponent is decreasing in $q$, therefore we could further bound it by
\begin{align*}
&\left(1-\frac{1-w}{(1-\eta)\Delta}\right)^{\frac{\left((1-\eta)\Delta+0.36\Delta+1\right)(1-\eta)\Delta}{\left((1-\eta)\Delta+1-w\right)(2-\eta-1/\Delta)}}\ge  \left(1-\frac{1-w}{(1-\eta)\Delta}\right)^{\frac{\left((1-\eta)\Delta+0.36\Delta\right)}{(2-\eta-1/\Delta)}}\\
&\ge \left(1-\frac{1-w}{(1-\eta)\Delta}\right)^{\frac{1.358\Delta}{(1.996)}}\ge 0.5053,
\end{align*}
where in the last inequality we plugged in $\Delta=500$ using the fact that the function is increasing in $\Delta$.
Similarly we obtain,
\begin{align*}
\left(1-\frac{w}{q- \Delta+1}\right)^{\tfrac{q}{(2-\eta-1/\Delta)}} &  
\ge \left(1-\frac{w}{(1-\eta)\Delta}\right)^{\frac{q(1-\eta)\Delta}{(2-\eta-1/\Delta)(q-\Delta+1)}}
\ge \left(1-\frac{w}{(1-\eta)\Delta}\right)^{\frac{(2-\eta)(1-\eta)\Delta^2}{(2-\eta-1/\Delta)((1-\eta)\Delta+1)}}
\\
&\ge\left(1-\frac{w}{(1-\eta)\Delta}\right)^{\frac{(2-\eta)\Delta}{(2-\eta-1/\Delta)}}
\ge \left(1-\frac{w}{(1-\eta)\Delta}\right)^{\frac{1.998\Delta}{1.996}}
\\
&\ge  0.9979.
\end{align*}

Therefore, we obtain that
\[
(f+1)\pr_{G}[\Phi(v)=j]\leq \frac{f+1}{0.504|L| }\leq \frac{1}{0.504(2-\eta-1/\Delta)}<1
 \]
using that $\Delta\ge 500$ and $\eta\le 0.002$, as desired.
\end{proof}

\section{Proof of Theorem~\ref{thm:main expanded} for \texorpdfstring{$\eta>0$}{eta>0}}\label{sec:proof eta>0}

\begin{proof}[Proof of Theorem~\ref{thm:main expanded} for $\bm{\eta>0}$]
Our goal is to expand on the proof for $\eta =0$ given in the previous section. 
To this end let us fix $\eta\in [0,0.002]$, positive integers $\Delta$ and $q\geq (2-\eta)\Delta$ and $w\in [0,1]$.
By the proof for the case $\eta=0$, we may assume that $\Delta\geq 500$ since $q$ is an integer.

The proof is again by induction on the number of free vertices; however, we add one more statement. We claim that there exists $\varepsilon_1,\varepsilon_2>0$ such that for any $(G,v)\in\mathcal{G}_{\Delta,q}^\bullet$ with free degree $d:=\deg_{\overline{G}}(v)\le \Delta-1$, any $i,j,k\in [q]$ and $\tilde w\in B(w,\varepsilon_1)$ we have 
\begin{enumerate}
    \item[(i)] $|R_{\overline{G},v;i,j}(w)-R_{\overline{G},v;i,j}(\tilde w)|\le \frac{(1-w)\deg_{\overline{G}}(v)+2/3}{\Delta}\varepsilon_2$,
    \item[(ii)] $\pr_{G^{+k},w}[\Phi(v)=j]\cdot |R_{\overline{G},v;i,j}(w)-R_{\overline{G},v;i,j}(\tilde w)|\le \varepsilon_2/\Delta$,
    \item[(iii)] $Z_G(\tilde w)\neq 0$.
\end{enumerate}
The proofs of (i) and (iii) are very similar to the proof for $\eta=0$. 
Moreover, the choice of $\varepsilon_1$ and $\varepsilon_2$ only depends on these parts. 
The more technical part will be the proof of (ii), where we will exploit local information around the vertex $v$ to show that either $\pr_{G^{+k},w}[\Phi(v)=j]$ is small or $|R_{\overline{G},v;i,j}(w)-R_{\overline{G},v;i,j}(\tilde w)|$ is small. 

The base case follows in exactly the same way as in the proof for $\eta =0$. In other words, if $\varepsilon_1$ is sufficiently small, then all three statements hold when there is exactly $1$ free vertex.

Next, let $(G,v)\in\mathcal{G}_{\Delta,q}^\bullet$ with $v$ of free degree $d\le \Delta-1$ with more than one free vertex. 
To prove statements (i) and (ii), we may of course assume by symmetry that $i=1$ and $j=q$.

Choose any ordering of the free neighbors of $v$ (to be specified further later).
Let $v_1,\ldots,v_d$ be the free neighbors of $v$ (in this order). 
Let for $i\in [d]$ $G_i$ be the graph obtained from $G$ via the telescoping procedure and recall that $(\overline{G}_i,v_i)$ denotes the graph obtained from $(G_i,v_i)$ by removing the pinned neighbors of $v_i$ in $G_i$.
Let $I\subseteq[d]$ be the set of indices for which $\pr_{G_i^{+1},w}[\Phi(v_i)=q]>\pr_{G_i^{+q},w}[\Phi(v_i)=1]$. Also let ${\bf c}^i={\bf c}_{G_i,v_i}$ the vector of blocked colors at $v_i$ in $G_i$ and $R^i_{j}(\tilde w)= R_{\overline{G}_i,v_i;j,q}(\tilde w)$, which by induction is well-defined.

We now prove all three statements, assuming they hold for partially $q$-colored graphs with fewer free vertices. 

\paragraph{Proof of (i).}
Following the same steps in the proof for $\eta=0$ we arrive at~\eqref{eq:bound on diff}, which states
\begin{align*}
|R_{\overline{G},v;1,q}(\tilde w)-R_{\overline{G},v;1,q}(w)|\le& \sum_{i=1}^{d}\underbrace{|F_{w,{\bf c}^i}({\bf R}^i(w))-F_{w,{\bf c}^i}({\bf R}^i(\tilde w))|}_{A_i}\\
&+\sum_{i=1}^{d}\underbrace{|F_{w,{\bf c}^i}({\bf R}^i(\tilde w))-F_{\tilde w,{\bf c}^i}({\bf R}^i(\tilde w))|}_{B_i}.
\end{align*}
Continuing the proof for $\eta=0$ we arrive at~\eqref{eq:bounding A_i}, which states 
\[
A_i\le |\langle {\bf P}_{{G}_i,v_i},{\bf x}\rangle|  +   \varepsilon_2\sup_{t\in[0,1]}\|{\bf P}_{{G}_i,v_i}-\nabla F_{w,{\bf c}^i}({{\bf y}_t})\|_1,
\]
where ${\bf x}={\bf R}^i(w)-{\bf R}^i(\tilde w)$ and ${{\bf y}_t}=t{\bf R}^i(w)+(1-t){\bf R}^i(\tilde w)$, whose infinity norms are both bounded by $\varepsilon_2$ by induction statement (i).
Thus, if $i\in I$, and if $\varepsilon_2$ is smaller than the $\delta$ from Proposition~\ref{prop:gradient and w} with $\varepsilon=\tfrac{1}{3\Delta^2}$, we can apply Proposition~\ref{prop:gradient and w} and Lemma~\ref{lem:replace by prob} to see that
we have 
\begin{align*}
A_i&\le (1-w)\pr_{G_i^{+1},w}[\Phi(v_i)=q]\cdot \max_{j=1\dots q-1}|R^i_{\overline{G}_i,v_i;j,q}(w)-R^i_{\overline{G}_i,v_i;j,q}(\tilde w)| + \varepsilon_2\tfrac{1}{3\Delta^2},
\end{align*}
which by induction statement (ii) is bounded by $(1-w)\varepsilon_2/\Delta +\tfrac{\varepsilon_2}{3\Delta^2}$.
The same reasoning gives us the same bound on $A_i$ when $i\in [d]\setminus I$.

To bound $B_i$ we use Proposition~\ref{prop:gradient and w} once more, to conclude that if $\varepsilon_1$ is smaller than the $\delta$ from Proposition~\ref{prop:gradient and w} with $\varepsilon=\tfrac{\varepsilon_2}{3\Delta^2}$ to conclude that for each $i=1,\ldots,d$, $|B_i|\le \tfrac{\varepsilon_2}{3\Delta^2}$.
Putting these bounds together we see that we obtain
\[
|R_{\overline{G},v;i,j}(\tilde w)-R_{\overline{G},v;i,j}(w)|\le \left(\frac{(1-w)d+2/3}{\Delta}\right)\varepsilon_2,
\]
thereby proving statement (i).

\paragraph{Proof of (ii).}
Now let us prove the second statement. Recall that we may assume that $i=1$ and $j=q$. 
Note that by Lemma~\ref{lem:prob basic} and the already proven item (i) we have 
\begin{align}
\pr_{G^{+k},w}[\Phi(v)=q]\cdot &|R_{\overline{G},v; 1,q}(w)-R_{\overline{G},v; 1,q}(\tilde w)|\le \nonumber
\\
    &\frac{w^{c_q}}{(1-\eta)\Delta}\cdot  \left( \frac{(1-w)d+2/3}{\Delta}\right)\varepsilon_2\label{eq:consquence of (i)}.
\end{align}
This implies that we may assume the following statements:
\begin{itemize}
\item[]\begin{itemize}
    \item[(A1)] $w\leq \eta$,
    \item[(A2)] $v$ has no pinned neighbor of color $q$,
    \item[(A3)] the free degree, $d$, of $v$ satisfies $d\ge (1-\eta)\Delta-2/3\geq 0.9966\Delta$.
\end{itemize}
\end{itemize}
Indeed, if any of (A1), (A2), (A3) does not hold, then we have that~\eqref{eq:consquence of (i)} is bounded by $\varepsilon_2/\Delta$, and we are done.

Let $H$ be the graph induced by $N_G(v)$. Next we claim that we may assume that the following statements hold:
\begin{itemize}
\item[]\begin{itemize}
    \item[(A4)] for $(1-\gamma)d\ge 0.98d$ neighbors of $v$ the color $q$ is blocked, 
    \item[(A5)] the average degree of $H$, $\overline{d}_H$, satisfies $\overline{d}_H=\beta d \geq 0.36d$.
\end{itemize}
\end{itemize}

Indeed, statements (A4) and (A5) follow because otherwise by Corollary~\ref{cor: few blocked} and Corollary~\ref{cor:sparse neighborhood} (which we can apply since by (A1) we may assume that $w\in [0,0.002]$) we would have $\pr_{G^{+k}}[\Phi(v)=q]\leq 1/\Delta$ and by the already proven item (i) we would obtain
\[
\pr_{G^{+k}}[\Phi(v)=q]\cdot |R_{\overline{G},v; 1,q}(w)-R_{\overline{G},v; 1,q}(\tilde w)|\le \frac{1}{d+1}\frac{(1-w)d+2/3}{\Delta}\varepsilon_2<\tfrac{1}{\Delta}\varepsilon_2.
\]

To proceed we will argue that these assumptions allow us to show the ratio difference $|R_{\overline{G},v; 1,q}(w)-R_{\overline{G},v; 1,q}(\tilde w)|$ has to be small. 
To do so we will take a closer look at the proof of item (i) of the induction hypothesis. 

We now specify the ordering of all the neighbors of $v$ as follows: first we have the pinned neighbors ($V_0$), then the vertices where $q$ is not blocked ($V_1$), and then the vertices where $q$ is blocked and not pinned ($V_2$). Within each group the vertices are ordered increasingly with respect to their degree in $H$.
We identify the labels of the vertices in $V_1\cup V_2$ (in an order preserving manner) with $[d]$.
Note that by (A3) we have $|V_0|\le 0.0034\Delta$ and by (A4) we have $|V_1|\le \gamma d$.
See Figure~\ref{fig:firstassumptions} for an illustration.

\begin{figure}[t]
    \centering
\includegraphics[width=0.7\textwidth]{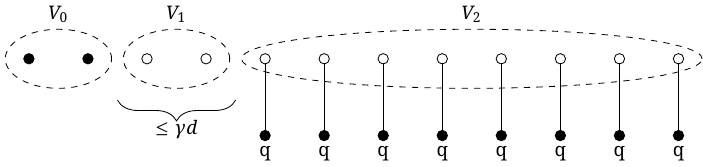}
  
      \caption{
        Illustration of the vertices in the neighborhood $H$ of vertex $v$, without the (possible) induced edges. 
        The set $V_0$ consists of the pinned vertices in $H$, $V_1$ consists of the vertices that do not have color $q$ blocked, and $V_2$ consists of the vertices where $q$ is blocked. 
        The vertices in $V_2$ are ordered from left to right, increasing in their degree in $H$. 
    }\label{fig:firstassumptions}
\end{figure}

As before, let $I\subseteq [d]$ be the set of indices of free vertices $v_i\in V(H)$, where $Z_{G_i,v_i}^q(w)>Z_{G_i,v_i}^1(w)$ (equivalently where $\pr_{G_i,w}[\Phi(v_i)=q]>\pr_{G_i,w}[\Phi(v_i)=1]$). 
Recall from the proof of item (i) that
\begin{align}
|R_{\overline{G},v;1,q}&(\tilde w)-R_{\overline{G},v;1,q}(w)| \nonumber
\\
\le&\sum_{i\in I}(1-w)\pr_{G_i^{+1},w}[\Phi(v_i)=q]\cdot \max_{j\in [q]}|R^i_{\overline{G}_i,v_i;j,q}(w)-R^i_{\overline{G}_i,v_i;j,q}(\tilde w)| \nonumber
\\
&+\sum_{i\in [d]\setminus I}(1-w)\pr_{G_i^{+q},w}[\Phi(v_i)=1]\cdot \max_{j\in[q]}|R^i_{\overline{G}_i,v_i;j,1}(w)-R^i_{\overline{G}_i,v_i;j,1}(\tilde w)|  \nonumber
\\
&+ \varepsilon_2\tfrac{2}{3\Delta}.\label{eq:expansion of ratio in eta>0}
\end{align}

Suppose that there are at least $\eta d$ indices $i$ in $I$ such that $q$ is blocked at $v_i$ (or at least $\eta d$ vertices in $V_1\cup V_2$ where $1$ and $q$ are blocked). Note that for these vertices $\pr_{G_i^{+1},w}[\Phi(v_i)=q]\le \frac{w}{q-\Delta(1-w)}$ (resp. $\pr_{G_i^{+q},w}[\Phi(v_i)=1]\le \frac{w}{q-\Delta(1-w)}$ and $\pr_{G_i^{+1},w}[\Phi(v_i)=q]\le \frac{w}{q-\Delta(1-w)}$). Then using that $w\leq \eta\leq 0.002$ by (A1), induction item (i) for these $\eta d$ vertices and item (ii) for the remaining vertices in~\eqref{eq:expansion of ratio in eta>0} and Lemma~\ref{lem:prob basic} to bound $\pr_{G^{+k},w}[\Phi(v)=q]$ by $\tfrac{1}{q-\Delta}$, we obtain that
\begin{align*}
\pr_{G^{+k},w}[\Phi(v)=q]\cdot &|R_{\overline{G},v;1,q}(\tilde w)-R_{\overline{G},v;1,q}(w)|
\\
&\le \frac{1}{q-\Delta}\left(\eta(1-w)d\frac{w}{(1-\eta+w)\Delta} + (1-\eta)d\frac{1-w}{\Delta}+\tfrac{2}{3\Delta}\right)\varepsilon_2
\\
&\le \frac{1}{q-\Delta}\left(\eta(1-w)\frac{w}{(1-\eta+w)} + (1-\eta)(\Delta-1)\frac{1-w}{\Delta}+\tfrac{2}{3\Delta}\right)\varepsilon_2
\\
&\le \frac{1}{(1-\eta)\Delta}\left(\eta\frac{w}{1-\eta} + \left(1-\eta\right)(1-w)\right)\varepsilon_2
\\
&\le \frac{\varepsilon_2}{\Delta}\left(\left(\frac{\eta}{(1-\eta)^2}-1\right)w+1\right)\leq \frac{\varepsilon_2}{\Delta}.
\end{align*}
So we may assume further that 
\begin{itemize}
    \item[] 
    \begin{itemize}
        \item[(A6)] at most $\eta d$ vertices in $I$ have color $q$ blocked, that is $|I\cap V_2|\leq \eta d$,
        \item[(A7)] at most $\eta d$ free neighbors of $v$ have both color $1$ and color $q$ blocked.
    \end{itemize}
\end{itemize}
See Figure~\ref{fig:secondassumptions} for an illustration.

\begin{figure}[t]
    \centering
\includegraphics[width=0.7\textwidth]{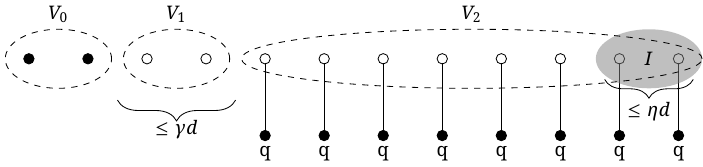}
   
\caption{
    Illustration as in Figure \ref{fig:firstassumptions}, with the additional detail that the set $I$ consists of the vertices $v_i$ such that $\pr_{G_i^{+1},w}[\Phi(v_i)=q]>\pr_{G_i^{+q},w}[\Phi(v_i)=1]$. 
    Note that the set $I$ is shown on the right for illustrative purposes, but it could be located anywhere and does not have to be concentrated on the right.
}\label{fig:secondassumptions}

\end{figure}

Now let us choose the largest index $i^\star$ such that $v_{i^\star}\in [d]\setminus I$ and such that $v_{i^\star}$ has  \\$(1-\gamma')\deg_{\overline{G}_{i^\star }}(v_{i^\star})\ge 0.86\deg_{\overline{G}_{i^\star }}(v_{i^\star})$ neighbors with $1$ blocked in $\overline{G}_{ i^\star}$ if this exists, otherwise let $i^\star=0$. We then have the following claim.

\begin{claim}\label{claim:final claim}
Under assumptions (A1)--(A7), there are $\rho d\ge 0.1 d$ indices in $N_{\overline{G}}(v)\setminus I$ that are larger than $i^\star$.     
\end{claim}

Before proving the claim, let us first show that it concludes the proof.
Indeed, if $j>i^\star$ such that $v_j \in N_{\overline{G}}(v)\setminus I $, then $v_j$ has at most $0.86\deg_{\overline{G}_j}(v_j)$ neighbors with color $1$ blocked in $\overline{G}^{+q}_j$ and therefore by Corollary~\ref{cor: few blocked} we have
\begin{equation}\label{eq:better bound prob  in eta>0}
     \pr_{G_j^{+q},w}[\Phi(v_j)=1]\le \frac{0.977}{\deg_{\overline{G}_j}(v_j)+1}.
 \end{equation}
This implies that by bounding $\pr_{G^{+k},w}[\Phi(v)=q]$ with $\tfrac{1}{q-\Delta}$ using Lemma~\ref{lem:prob basic} and plugging~\eqref{eq:better bound prob  in eta>0} into~\eqref{eq:expansion of ratio in eta>0} for these $\rho d\geq .1d$ vertices $j>i^\star$ in combination with induction item (i) and item (ii) for the remaining vertices, we obtain
\begin{align*}
\pr_{G^{+k},v}[\Phi(v)=q]\cdot |R_{\overline{G},v;1,q}(\tilde w)-R_{\overline{G},v;1,q}(w)|
&\le \frac{1}{q-\Delta}\left(0.1d(1-w)\tfrac{0.977}{\Delta} \varepsilon_2 + 0.9d\tfrac{\varepsilon_2}{\Delta}+\tfrac{2\varepsilon_2}{3 \Delta}\right)
\\
&\le \frac{1}{q-\Delta}\left(0.0977 \varepsilon_2 + 0.9(\Delta-1)\tfrac{\varepsilon_2}{\Delta}+\tfrac{2\varepsilon_2}{3 \Delta}\right)
\\
&\le
\frac{1}{(1-\eta)\Delta}\left(0.9977 \varepsilon_2 \right)\\
&<\frac{\varepsilon_2}{\Delta},
\end{align*}
finishing the proof of item (ii).

\begin{subproof}[Proof of Claim~\ref{claim:final claim}.]
To prove the claim we argue by contraction and suppose that it is not true, that is we assume 
\begin{itemize}
    \item[] \begin{itemize}
        \item [(A8)] the number of indices in $N_{\overline{G}}(v)\setminus I$ that are larger than $i^\star$, $\rho d$, is smaller than $0.1d$ 
    \end{itemize}
\end{itemize}
See Figure~\ref{fig:thirdassumptions} for an illustration.

\begin{figure}[t]

    \centering
\includegraphics[width=0.7\textwidth]{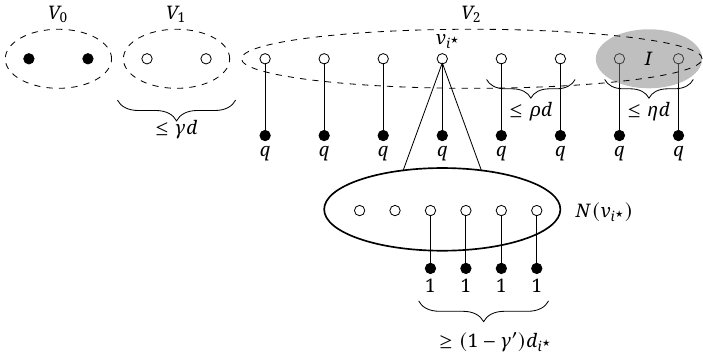}
 
     \caption{
        Illustration as in Figure \ref{fig:secondassumptions}, with the further detail that vertex $v_{i^\star}$ is the vertex in $V_2 \setminus I$ with the largest degree, such that at least $(1-\gamma')\deg_{\overline{G}_{i^\star}}$ of its neighbors have color $1$ blocked.
       For illustrative purposes, $N(v_{i^\star})$ has been drawn as disjoint from $H$. However, note that $N(v_{i^\star})$ could intersect with $H$.
    }\label{fig:thirdassumptions}
\end{figure}

Our plan is to show that this leads to a contradiction with assumption (A7).
First we show that (A8) (in combination with some of the previous assumptions) implies that 
\begin{itemize}
    \item[] \begin{itemize}
        \item [(A9)] the degree of $v_{i^\star}$ inside $H$ is at least $0.27\Delta$. 
    \end{itemize}
\end{itemize}

This follows by a density argument.
Since there are at most $\eta d$ indices in $I$ where $q$ is blocked  by (A6), this implies that there are at most $(\rho+\eta)d\leq 0.102 d$  indices in $V_2$ that are larger than $i^\star$.

To proceed let the degree of $v_{i^\star}$ in $H$ be denoted by $d_{i^\star}\le d -1$.
Then $d_{i^\star}$ is an upper bound on the degree of the vertices in $V_2$, whose index is at most $i^\star$ since inside $V_2$ the vertices are ordered by their degree in $H$.
In what follows we will upper bound the number of edges inside $H$.
Recall that by (A3) we have $d\geq 0.9966 \Delta$.
Note that the free neighbors of $v$ are the only vertices that can have a positive degree inside $H$. 
For the vertices inside $V_1$ and those inside $V_2$ that have a larger index than $i^\star$ we bound their degree in $H$ by $d-1$.
We have at most $\gamma d+ (\rho+\eta)d$ of these vertices. For the remaining vertices we have an upper bound of $d_{i^\star}$ on their degree. 
Therefore by (A5),
\[
\frac{1}{2}\beta d^2 \leq e(H)\le \frac{1}{2}\left( (\gamma + \rho+\eta)d(d-1)+(1-\gamma-\rho-\eta)d\cdot d_{i^\star}\right),
\]
and hence
\[
 d_{i^\star}\ge \frac{\beta d -(\gamma+\rho+\eta)(d-1)}{1-(\gamma+\rho+\eta)}\ge \frac{0.238}{0.878} d \geq 0.27\Delta,
\]
since $\beta\geq 0.36$, $\gamma\leq 0.02$, $\rho\leq 0.1$ (by (A8)), $\eta\leq 0.002$ and $d\ge 0.9966\Delta$.
This shows that (A9) holds.

The fact that the degree of $v_{i^\star}$ inside $H$ is large in combination with the telescoping procedure gives us information about how many neighbors of $v$ must have had color $1$ pinned before the telescoping procedure. 
Indeed, by the telescoping procedure only the neighbors of $v$ with an index larger than $i^\star$ will obtain a pinned neighbor with color $1$ in $G_{i^\star}$ (the others will receive a pinned neighbor with color $q$).
The idea is that since $v_{i^\star}$ has many neighbors inside $H$ (which we recall to be the graph induced by the free neighbors of $v$) and many neighbors inside $G_{i^\star}$ for which color $1$ is blocked, there must be an overlap of these two sets that does not include neighbors of $v$ with a larger index than $i^\star$ and hence these vertices must have had color $1$ blocked before applying the telescoping procedure to $(G,v)$.

We will now make this formal.
Let $A=N_H(v_{i^\star})$ and $B=\{u\in N_{\overline{G}_{i^\star}}(v_{i^\star})~|~ \textrm{1 is blocked at $u$}\}$. Then by (A9),
\begin{align*}
|A\cap B|&\ge d_{i^\star}+|B|-\deg_{\overline{G}_{i^\star}}(v_{i^\star})
\\
&\ge 0.27\Delta-\gamma'\Delta+1\geq 0.13\Delta.
\end{align*}
since $\gamma'\leq 0.14$ by our choice of $i^\star$.
This implies that $v_{i^\star}$ has at least $0.13\Delta$ neighbors inside $H=N_G(v)$ where color $1$ is blocked in the partially $q$-colored graph $G_{i^\star}$. 
The only pinned vertices of color $1$ in $G_{i^\star}$ that arose from the telescoping procedure are the new neighbors of the $v_i$ with $i>i^\star$. 
Since there are at most $0.102\Delta$ vertices $v_i$ with $i>i^\star$, at least $0.028\Delta$ vertices inside $H$ must already have had color $1$ blocked before the telescoping procedure was applied to $(G,v)$.

Therefore, since $|V_2|\geq (1-\gamma)d> 0.976 \Delta$ there are at least $0.004\Delta\ge 0.004 d$ free neighbors of $v$ where both color $1$ and color $q$ is blocked. 
This contradicts (A7) (since $\eta\leq 0.002$) and concludes the proof of the claim. 
\end{subproof}

Statement (iii) follows from statement (i) via Lemma~\ref{lem:lowerbound_wtilde_sum} in an identical manner as in our proof for the $\eta=0$ case.
This finishes the proof.
\end{proof}

\begin{remark}\label{rem:choice of eps1 eta>0}
As in Remark~\ref{rem:choice of eps1} we note that by the alternative proof of Proposition~\ref{prop:gradient and w} given in the appendix we have that $\varepsilon_1$ can be chosen to be $C'\Delta^{-4}$ for some constant $C'>0$.
\end{remark}

\section{Further remarks and conclusions }\label{sec:conclusion}
The main conceptual ingredient that allowed us to break the $q=2\Delta$ barrier is to carefully use the local structure of the neighborhood of a vertex to bound the marginal probability of the root vertex in combination with information about the log-ratios of vertices at distance at most $2$ of the root vertex.
It is tempting to do a more systematic analysis of the behavior of the log-ratios at the vertex $v$ in terms of the log-ratios at distance $2$ of $v$, but it is not clear to us how to make use of this extra information.

Clearly, our bound on $\eta$ can be improved once one has better bounds in Corollaries~\ref{cor: few blocked} and~\ref{cor:sparse neighborhood}. It is not unlikely that these corollaries can be somewhat improved, but we leave this for possible future work.
We suspect that a more substantial improvement on our bound on $\eta$ could be obtained if instead of controlling the $\ell_\infty$ norm of the vector $\left(R_{G,v;j,q}(\tilde w)-R_{G,v;j,q}(w)\right)_{j\in [q-1]}$ one can perhaps simultaneously control its $\ell_1$- and $\ell_\infty$ norm.

As remarked in the introduction for reasons of clarity we focused on the Potts model on bounded degree graphs. We next collect some remarks about extending our approach to different settings, that we didn't pursue here, because it would distract the focus from the main ideas presented in the current paper. 

\paragraph{Multivariate Potts model.}
One possible extension is to allow for a graph $G=(V,E)$ of maximum degree $\Delta$ a vector of edge weights $(w_e)_{e\in E}$ and thereby transform the associated partition function to the multiaffine polynomial
\begin{align*}
    Z_G(q,(w_e)_{e\in E})=\sum_{\phi:V\to [q]}\prod_{\substack{e=uv\in E\\\phi(u)=\phi(v)}}w_e.
\end{align*}
Following the approach from~\cite{BDPR21} for the multivariate setting, it should not be difficult to extend our main result to this setting.

\paragraph{List coloring.}
Another natural extension to consider is list colorings, i.e. equipping for a graph $G=(V,E)$ each vertex $v\in V$ with a list of colors $L(v)\subset \mathbb{N}$ and defining 
\begin{align*}
    Z_G((L_v)_{v\in V},w)=\sum_{\phi\in \prod_{v\in V}L(v)}\prod_{\substack{e=uv\in E \\ \phi_u=\phi_v}}w.
\end{align*}
Since all of our bounds only really depend on the number of available colors, it should again not be difficult to extend our main result to list coloring setting, provided each list $L(v)$ satisfies $|L(v)|\geq (2-\eta)\Delta(G)$. 

\paragraph{Triangle free graphs.}
As remarked in Remark~\ref{rem:marginal}, our proof for the $\eta=0$ case can recover the result of Liu, Sinclair, and Srivastava for triangle free graphs~\cite{LSS2Delta}. It would be interesting to see if the ideas that we employed for the $\eta>0$ case can somehow be used to also improve their bounds for triangle free graphs. 
It is not immediately clear how to do that, since the neighborhood of a vertex in a triangle free graphs is always an independent set, which has already been taken into account in bounds on the marginal probabilities. Possibly one has to take into account vertices at larger distance from the root vertex.
Another interesting question is to see if better bounds can be obtained if one assumes stronger bounds on the girth of the graph.

\paragraph{Small $\Delta$.}
For small values of $\Delta$, we can obtain better bound for the marginal probabilities of the root vertex by examining all possible neighborhood structure. In combination with our inductive proof this may lead to improved zero-free regions.
For example, when $\Delta=3$ the free-neighbors of a vertex $v$ of free degree at most $2$ could form an independent set of size at most $2$ or induces an edge. 
Examining the exact marginal bounds one can easily obtain a zero-free region containing the interval $[0,1]$ for the anti-ferromagnetic Potts model when $q\ge 5$. We suspect that with sufficient additional effort it should be possible to extend this to $q=4$, but it would involve investigation of local neighborhoods of size bigger than $1$. 
We note that~\cite{Luetalfptasforcubic} has bounds on these marginal probabilities in the $q=4$ and $\Delta=3$ case. 
However, in their setting they don't have an additional pinned neighbor of the root, and as such their bounds cannot be used directly in our setting. 
Additional effort is needed to see if their bounds could be helpful in our setting.

\section*{Acknowledgments}
We thank the anonymous referees for their constructive feedback.

\printbibliography
\clearpage

\appendix

\section{A direct proof of Proposition~\ref{prop:gradient and w}}\label{appendix:additional computations}

We restate an updated version of the proposition here for convenience.
\begin{proposition}\label{prop:gradient updated} 
Let $\alpha>0$ and let $q,\Delta$ be positive integers such that $q\geq (1+\alpha) \Delta+1$.
Let $C_1(\alpha)$ be as in Lemma~\ref{lem:lowerbound_wtilde_sum}.
There is a constant $C_2(\alpha)>0$, such that for any $\varepsilon\in (0,1)$ there exists a $\delta=\min\left\{ \frac{\pi}{8},\frac{C_1(\alpha)}{\Delta}, C_2(\alpha)\varepsilon, \frac{\varepsilon}{8C_2(\alpha)}  \right\} $  
such that the following holds.
Let $(G,v)\in\mathcal{G}_{\Delta,q}^\bullet$, and let $w\in[0,1]$.
Let ${\bf R}\in \mathbb{R}^{q-1}$ be the vector defined by
\[
R_{j}=R_{\overline{G},v;j,q}(w).
\]  
Then
\begin{itemize}
    \item[(i)]
    if ${\bf x} \in \mathbb{C}^{q-1}$ and $\|{\bf R}-{\bf x}\|_\infty\le \delta$
\[
\|{\bf P}_{G,v}(w)-\nabla F_{w,{\bf c}}({\bf x})\|_1\le \varepsilon;
\]
\item[(ii)] if ${\bf x}\in\mathbb{C}^{q-1}$ and $\|{\bf R}-{\bf x}\|_\infty\le \delta$ and $|\tilde w-w|\leq \delta$, then 
\[
    |F_{w,{\bf c}_{G,v}}({\bf x})-F_{\tilde w,{\bf c}_{G,v}}({\bf x})|\le \varepsilon.
\]
\end{itemize}    
\end{proposition}

\begin{proof}

Let $M=\frac{(1+ \alpha)e^{1/\alpha}}{\alpha}$.
Recall that by Lemma~\ref{lem:prob basic} and Lemma~\ref{lem:prob basic lower}, and by the fact that $q/(q-\Delta)$ and $(q-\Delta)/q$ are respectively decreasing  and increasing for $q>\Delta$, we have that 
 \begin{equation*}
 \frac{\alpha \Delta +1}{\left((1+\alpha)\Delta +1\right) e^{1/\alpha}} \leq   \frac{q-\Delta}{qe^{1/\alpha}} \leq \exp(R_j)\le  \frac{qe^{1/\alpha}}{q-\Delta} \leq \frac{ \left((1+\alpha)\Delta +1\right) e^{1/\alpha} }{\alpha \Delta +1} ,
 \end{equation*}
and therefore in particular 
\begin{equation*}
     M^{-1} \le \exp(R_j)\le  M.
\end{equation*}

Next, let $C_2(\alpha) =\frac{\alpha^3 }{16 (1+\alpha)^3 e^{2/\alpha}}$ and note that
\begin{align*}
   \delta \leq C_2(\alpha)\varepsilon   \leq   \varepsilon\frac{\alpha }{16M^2 (1+\alpha)}  \leq\varepsilon\frac{\alpha \Delta }{16M^2 (1+\alpha)\Delta}  \leq   \frac{(q-\Delta-1)\varepsilon}{16M^2(q-1)} .
\end{align*}
We will now show that the conclusion of the proposition holds with our choice of $\delta$.

 Let $\bf{R}$ be as defined in the proposition, and let $\bm{x} \in \C^{q-1}$ such that $\|\bm{R}- \bm{x}\|_\infty < \delta.$
 We use shorthand notation $P$ and $Q$ for $P_{\bm{c}}(w, \bm{R})$ and $Q_{\bm{c}}(w, \bm{R})$, which are defined in Section~\ref{subsec:outline}.
Since $e^{x_j}$ is contained in a cone of angle at most $2\delta$ centered at the real axis for all $j \in [q-1]$, and since $\delta$ is small enough, it follows by Lemma~\ref{Lem:Barvinok} that
\begin{align*}
    |P(\bm{x})| &\geq \cos(\delta) \left(w^{c_1+1}|e^{x_1 }|+ \sum_{j=2}^{q-1}w^{c_j}|e^{x_j}|\right).\\
\end{align*}
Here $P(\bm{x})$ is shorthand for $P_{\bm{c}}(w,\bm{x})$.
Note that 
\[
|e^{x_j}| = |e^{R_j}e^{x_j-R_j}| = |e^{R_j}| |e^{\Re(x_j)-\Re(R_j)}| \geq e^{ R_j}e^{-\delta}, 
\] 
and thus 
\[
 |P(\bm{x})|  \geq \cos(\delta)e^{-\delta}\left(w^{c_1+1}e^{R_1} +\sum_{j=2}^{q-1}w^{c_j}e^{R_j}\right)= \cos(\delta)e^{-\delta}P.
\]

Simply by applying the triangle inequality and observing that $|e^{x_j}| \leq e^{ R_j}e^{\delta}$, we can also conclude that
\[
|P(\bm{x})| \leq e^\delta P.
\]
There exists a complex number $\xi_j$ of absolute value at most $\delta$ such that $e^{R_j(t)}=e^{\xi_j}e^{R_j}$ for all $j=1, \ldots q-1$. 
Using this and the assumption that $e^{\delta}\leq 1+2\delta$, which is true since $\delta< 0.5 <\log(2)$, we see that 
\begin{align*}
   & \left|\frac{e^{x_j}}{P(\bm{x})}-\frac{e^{R_j}}{P}\right|=\frac{|Pe^{x_j}-e^{R_j}P(\bm{x})|}{|P(\bm{x})|P}
   \leq  \frac{e^{R_j}|Pe^{\xi_j}-P(\bm{x})|}{P|P(\bm{x})|}
   \\
   &\leq \frac{e^{R_j} P\max_{k}|e^{\xi_j}-e^{\xi_k}|}{P|P(\bm{x})|}
\leq\frac{e^{R_j}e^{\delta} 2\max_k|e^{\xi_k}-1|}{\cos(\delta)P} 
\\
&  \leq    \frac{e^{R_j}e^\delta 2(e^\delta -1)}{P}  \leq \frac{4\delta e^{\delta}}{\cos(\delta)}\frac{e^{R_j}}{P} \leq \frac{8 \delta M^2}{(q-\Delta-1)} < \frac{\varepsilon}{2(q-1)}.
\end{align*}
In the penultimate inequality, we use $e^\delta/\cos(\delta) \leq 2$ for $\delta$ smaller than $0.5$.
Furthermore, the $q-\Delta-1$ comes from the fact that we have at least $q-\Delta$ zero entries for the $\bm{c}$ vector, but the first term of $P$ has a factor $w^{c_1 + 1}$.

In an analogous way, we obtain for the same choice of $\delta$ that
\[
\left|\frac{e^{x_j}}{Q(\bm{x})}- \frac{e^{R_j}}{Q}  \right| < \frac{\varepsilon}{2(q-1)}.
\]

Combining everything gives us the following:
\begin{align*}
    \|\bm{P}_{G,v} - \nabla F_{w,c}(\bm{x}) \|_1 & \leq   \|\bm{P}_{G,v} -  \nabla F_{w,c}(\bm{R})\|_1  + \| \nabla F_{w,c}(\bm{R}) -\nabla F_{w,c}(\bm{x}) \|_1\\
    &=\| \nabla F(\bm{R}) -\nabla F(\bm{x}) \|_1\\
    & \leq  \left| w^{c_1+1} \frac{e^{R_1}}{P} - w^{c_1+1} \frac{e^{x_1}}{P(\bm{x})} \right| +  \left| w^{c_1} \frac{e^{R_1}}{Q} - w^{c_1} \frac{e^{x_1}}{Q(\bm{x})} \right|\\
    &+ \sum_{j=2}^{q-1}\left(   \left| w^{c_j} \frac{e^{R_j}}{P} - w^{c_j} \frac{e^{x_j}}{P(\bm{x})} \right| +  \left| w^{c_j} \frac{e^{R_j}}{Q} - w^{c_j} \frac{e^{x_j}}{Q(\bm{x})} \right| \right) \\
    &<\frac{ w^{c_1+1} \varepsilon}{2(q-1)} + \frac{ w^{c_1}\varepsilon}{2(q-1)}  + \sum_{j=2}^{q-1} w^{c_j} \left( \frac{\varepsilon}{2(q-1)}  +  \frac{\varepsilon}{2(q-1)} \right)\\
    &\leq \varepsilon,
\end{align*}
proving (i).

For (ii),  we consider the following
\begin{align*}
    |F_{w,{\bf c}_{G,v}}({\bf x})-F_{\tilde w,{\bf c}_{G,v}}({\bf x})|&= \left|\log\left( \frac{P(w, \bm{x})}{Q(w, \bm{x})}\right)  - \log\left(\frac{P(\tilde{w}, \bm{x})}{Q(\tilde{w}, \bm{x})}\right)  \right| \\
    &= \left|\log\left( \frac{P(\tilde{w}, \bm{x})}{P(w, \bm{x})}  \cdot \frac{Q(w, \bm{x})}{Q(\tilde{w}, \bm{x})} \right)  \right|\\
    & \leq \left|\log\left( \frac{P(\tilde{w}, \bm{x})}{P(w, \bm{x})} \right) \right| + \left| \log\left( \frac{Q(w, \bm{x})}{Q(\tilde{w}, \bm{x})} \right)  \right|.
\end{align*}
Since $\log(z) = \log|z| + i \arg(z)$, our goal is to bound the absolute value and arguments of $\frac{P(\tilde{w}, \bm{x})}{P(w, \bm{x})},\frac{Q(w, \bm{x})}{Q(\tilde{w}, \bm{x})}$.

By remarking that
\[
|\tilde{w}^{c_j} -w^{c_j} | \leq |\tilde{w} - w| |c_j| \max(|w|,|\tilde{w}|^{c_j}) < \delta( 1+ \delta)^{c_j}c_j,
\]
we obtain
 \begin{align*}
     |Q(w, \bm{x}) - Q(\tilde{w}, \bm{x})|& = \left|  \sum_{j=1}^{q-1} \tilde{w}^{c_j} e^{x_j} + \tilde{w}^{c_q+1} - \sum_{j=1}^{q-1} w^{c_j} e^{x_j} - w^{c_q+1}  \right|\\
   &= \left|   \sum_{j=1}^{q-1} (\tilde{w}^{c_j} - w^{c_j} ) e^{x_j} + (\tilde{w}^{c_q+1} - w^{c_q+1})  \right|\\
   & < \sum_{j=1}^{q-1} c_j \delta(1 + \delta)^{c_j}  e^\delta M + (c_{q}+1) \delta  (1+ \delta )^{c_q+1}   e^\delta M\\
   &\leq  \Delta \delta(1+ \delta)^{\Delta}M e^\delta\leq 2\Delta\delta M e^{C_1(\alpha)},
\end{align*}
since $\delta \leq \frac{C_1(\alpha)}{\Delta}$.
Since by Lemma~\ref{lem:lowerbound_wtilde_sum} it follows that $|Q(\tilde{w}, \bm{x})| \ge C(\alpha)\Delta$, we obtain
\begin{align*}
    \left| \frac{Q(w, \bm{x})-Q(\tilde{w}, \bm{x})}{Q(\tilde{w}, \bm{x})}\right| & \leq  \frac{2 \delta M  e^{ C_1(\alpha) }}{C(\alpha)} \leq   \frac{2 \delta  e^{C_1(\alpha) } e^{1/\alpha} }{C(\alpha)} \frac{1+ \alpha}{\alpha} =  C_2(\alpha)\delta.
\end{align*}

Let us write $z=\frac{Q(w, \bm{x})-Q(\tilde{w}, \bm{x})}{Q(\tilde{w}, \bm{x})}$. 
Then, since
\begin{align*}
     |\arg(1+z)| = \left| \arctan\left( \frac{\Im(z)}{1+\Re(z)} \right)\right| \leq \left| \frac{\Im(z)}{1+\Re(z)} \right|  \leq \frac{|z|}{1-|z|},
\end{align*}
it follows that
\[
 \left|\log\left( \frac{Q({w}, \bm{x})}{Q(\tilde w, \bm{x})} \right) \right|  \leq \left|\log|1+z| \right| + \frac{|z|}{1-|z|}\leq \log( 1+ C_2(\alpha)\delta) + \frac{C_2(\alpha)\delta}{1- C_2(\alpha)\delta}.
\]
It follows in a similar fashion that, 
\[
\left| \log\left( \frac{P(\tilde w, \bm{x})}{P({w}, \bm{x})} \right)  \right| \leq \log( 1+ C_2(\alpha)\delta) + \frac{C_2(\alpha)\delta}{1- C_2(\alpha)\delta}.
\]
Since $C_2(\alpha)\delta \leq \tfrac{\varepsilon}{8}$,
we have $C_2(\alpha)\delta/\left(1-C_2(\alpha)\delta\right) \leq 2C_2(\alpha)\delta$.
Thus, we may conclude
\begin{align*}
    |F_{w,{\bf c}_{G,v}}({\bf{x} })-F_{\tilde w,{\bf c}_{G,v}}({\bf{x}})| &  \leq \left|\log\left( \frac{P(\tilde{w}, \bm{x})}{P(w, \bm{x})} \right) \right| + \left| \log\left( \frac{Q(w, \bm{x})}{Q(\tilde{w}, \bm{x})} \right)  \right|\\
    &\leq 2\log( 1+ C_2(\alpha)\delta)  + 4C_2(\alpha)\delta\\
    & \leq   2\log(\exp(\varepsilon/8))  + \frac{\varepsilon}{4}\\
    & < \varepsilon,
\end{align*}
as desired.
 \end{proof}

\end{document}